\documentclass[english]{extarticle}
\usepackage{mathptmx}
\usepackage[T2A,T1]{fontenc}
\usepackage[latin9]{inputenc}
\usepackage{geometry}
\usepackage{babel}
\usepackage{float}
\usepackage{amsmath}
\usepackage{amsthm}
\usepackage{amssymb}
\usepackage{graphicx}
\usepackage[unicode=true,pdfusetitle,
 bookmarks=true,bookmarksnumbered=false,bookmarksopen=false,
 breaklinks=false,pdfborder={0 0 0},pdfborderstyle={},backref=false,colorlinks=false]
 {hyperref}

\makeatletter

\DeclareRobustCommand{\cyrtext}{%
  \fontencoding{T2A}\selectfont\def\encodingdefault{T2A}}
\DeclareRobustCommand{\textcyr}[1]{\leavevmode{\cyrtext #1}}

\providecommand{\tabularnewline}{\\}

\numberwithin{equation}{section}
\numberwithin{figure}{section}
\theoremstyle{plain}
\newtheorem{thm}{\protect\theoremname}[section]
\theoremstyle{remark}
\newtheorem{rem}[thm]{\protect\remarkname}
\theoremstyle{plain}
\newtheorem{lem}[thm]{\protect\lemmaname}
\theoremstyle{definition}
\newtheorem{defn}[thm]{\protect\definitionname}
\theoremstyle{definition}
\newtheorem{example}[thm]{\protect\examplename}
\theoremstyle{plain}
\newtheorem{assumption}[thm]{\protect\assumptionname}

\@ifundefined{date}{}{\date{}}
\makeatother

\providecommand{\assumptionname}{Assumption}
\providecommand{\definitionname}{Definition}
\providecommand{\examplename}{Example}
\providecommand{\lemmaname}{Lemma}
\providecommand{\remarkname}{Remark}
\providecommand{\theoremname}{Theorem}

\begin{document}
\title{Path weighting method for spot sensitivities and its variance reduction}
\author{Xuan Liu\thanks{Nomura Securities, Hong Kong SAR. Email: chamonixliu@163.com}
\enspace{}and Michel Gauthier\thanks{Nomura Securities, Tokyo, Japan. Email: michel.gauthier@nomura.com}}
\maketitle
\begin{abstract}
In this paper, we study the computation of sensitivities with respect
to spot of path dependent financial derivatives by means of path weighting.
We propose explicit path weighting formula and variance reduction
adjustment in order to address the large variance happening when the
first simulation time step is small. We also propose a covariance
inflation technique to addresses the degenerator case when the covariance
matrix is singular. The stock dynamics we consider is given in a general
functional form, which includes the classical Black--Scholes model,
the implied distribution model, and the local volatility model.
\end{abstract}

\section{Introduction\label{sec:Introduction}}

The computation of greeks for exotic financial derivatives is critical
to the pricing and hedging of financial derivatives. The most widely
used method for greeks computation is the finite difference method
(combined with Monte Carlo simulation). It is well known that the
convergence of greeks becomes very slow when there is discontinuity
in the payoff function. In such case, the greek--spot graph is usually
spiky and noisy. There has been research in alternative numerical
methods such as pathwise sensitivities, likelihood ratio methd (cf.
\cite{Gla04}), and Malliavin weighting method (which expresses greeks
as an expectation of the payoff with paths weighted by a stochastic
integral). For many cases of interest, the likelihood ratio method
and Malliavin weighting method are equivalent. Therefore we prefer
calling them path weighting method. To the best knowledge of the authors
of the current paper, the path weighting method is not used in industry
is mainly due to two reasons: (i) the path weighting formula in existing
academic research are not explicit enough for diffusion models used
in industry; and (ii) the convergence of path weighting sensitivities
is not as good as finite difference. In this paper, we address the
first problem by deriving the path weighting formula for discrete
dynamics, which computes the full Delta vector and Gamma matrix as
a single pricing. For the second problem, we propose a variance reduction
adjustment, which improves the convergence significantly at the cost
of one more pricing for full Delta vector and two more pricing for
full Gamma matrix. We also derive asymptotic variance order for the
adjustment formula under different regularity assumptions made on
the payoff function. Finally, we propose the covariance inflation
technique, which not only addresses singular covariance matrices,
but also further reduces the variance of the MC estimators.

Code for all numerical tests is available in the open source library
\texttt{pwsen}\footnote{Source code available at \href{https://github.com/liuxuan1111/pwsen}{https://github.com/liuxuan1111/pwsen}}. 

\section{\label{sec:-1}The Path Weighting Formula}

\subsection{\label{subsec:-1}General Case}

Denote the stock prices at time $t$ by a column vector
\[
S_{t}=\big(S_{t}^{(1)},\dots,S_{t}^{(n)}\big)^{\text{T}}.
\]
Consider a Markovian diffusion model for stock dynamics
\begin{align*}
dS_{t} & =a(t,S_{t})dt+b(t,S_{t})dW_{t},\\
d\langle W^{(i)} & ,W^{(j)}\rangle=\Sigma_{ij}dt,
\end{align*}
where $a(t,s)$, $b(t,s)$ are $\mathbb{R}^{n\times n}$-valued functions,
and $\Sigma$ is the correlation matrix. For Monte--Carlo simulation,
the Euler--Maruyama discretisation of the above SDE is used
\[
\begin{aligned}S_{t_{i+1}} & =S_{t_{i}}+\Delta t_{i+1}a(t_{i},S_{t_{i}})+\sqrt{\Delta t_{i+1}}b(t_{i},S_{t_{i}})\rho Z_{i+1},\\
S_{t_{0}} & =s,
\end{aligned}
\]
where $s>0$ is the spot price, $\Delta t_{i+1}=t_{i+1}-t_{i}$, and
$\{Z_{i}\}_{i>0}$ is a sequence of i.i.d. standard normal random
variables, and 
\begin{equation}
\rho=Q\Lambda\label{eq:-2}
\end{equation}
 with $Q^{\text{T}}\Sigma Q=\Lambda^{2}$ being the eigenvalue decomposition
of $\Sigma$. 

We may write the above dicretised dynamics in a slightly more general
form
\begin{equation}
\begin{aligned}S_{t_{i+1}} & =H(t_{i+1},\Delta t_{i+1},S_{t_{i}},\sqrt{\Delta t_{i+1}}Z_{i+1})\\
S_{t_{0}} & =s,
\end{aligned}
\label{eq:}
\end{equation}
or simply
\begin{equation}
S_{t_{i+1}}=H_{i}(S_{t_{i}},\sqrt{\Delta t_{i+1}}Z_{i+1}).\label{eq:-1}
\end{equation}
Throughout this paper, we assume that the functional $H_{0}(s,x)$
is smooth in $s$ and $x$, and that the matrix $\partial_{x}H_{0}$
is invertable. As we shall see, the non-singularity of $\partial_{x}H$
requires non-singular correlation matrix $\Sigma$.
\begin{rem}
The functional $H_{i}(s,x)$ in recursion (\ref{eq:}) is general
enough to include the following cases:
\end{rem}

\begin{itemize}
\item Implied distribution model.
\item Euler--Maruyama discretisation of local volatility model.
\item Piecewise Black--Scholes approximation
\[
dS_{t}=a(t_{i},S_{t_{i}})S_{t}dt+\sigma(t_{i},S_{t_{i}})S_{t}dW_{t},\;t\in[t_{i},t_{i+1}),
\]
or equivalently,
\[
S_{t_{i+1}}=S_{t_{i}}\cdot\exp\Big[\Big(a(t_{i},S_{t_{i}})-\frac{\sigma(t_{i},S_{t_{i}})^{2}}{2}\Big)\Delta t_{i+1}+\sigma(t_{i},S_{t_{i}})\sqrt{\Delta t_{i+1}}Z_{i+1}\Big].
\]
\end{itemize}
Let $F(S_{t_{1}},\dots,S_{t_{N}})$ be the payoff of a financial derivatives
(or its approximation corresponding to the SDE dscretisation). Its
price is given by the expectation under the risk neutral probability
measure
\[
f(s)=\mathbb{E}[F(S_{t_{1}},\dots,S_{t_{N}})|S_{t_{0}}=s].
\]
We are interested in numerical calculation of $f^{\prime}(s)$. 

Assume that $0=t_{0}<t_{1}<\cdots<t_{N}=T$ are the times for SDE
discretisation, and that evolution of stock prices are given by the
functional form
\begin{equation}
\begin{aligned}S_{t_{i+1}}^{(j)} & =H_{i}^{(j)}(S_{t_{i}},\sqrt{\Delta t_{i+1}}Z_{i+1}),\;0\le i<N,\\
S_{t_{0}}^{(j)} & =s_{j},
\end{aligned}
\label{eq:-14}
\end{equation}
where $H_{i}^{(j)}(s,x)$ have continuous second derivatives with
respect to $x$, $\Delta t_{i+1}=t_{i+1}-t_{i}$, $s=(s_{1},\dots,s_{n})^{\text{T}}$
is the stock spot (column) vector, and 
\[
Z_{i+1}=\big(Z_{i+1}^{(1)},\dots,Z_{i+1}^{(n)}\big)^{\text{T}},\;0\le0<N,
\]
are i.i.d. standard normal random (column) vectors.

The theorem below presents path weighting formula. A proof of Theorem
\ref{thm:-3} is given in Appendix \ref{subsec:-2}.
\begin{thm}
\label{thm:-3}Let $F(S_{t_{1}},\dots S_{t_{N}})$ be a measurable
function with at most polynomial growth at infinity, and 
\[
f(s)=\mathbb{E}[F(S_{t_{1}},\dots,S_{t_{N}})|S_{t_{0}}=s].
\]
Denote
\begin{equation}
J(s,x)=(\partial_{x}H_{0})^{-1}(\partial_{s}H_{0}).\label{eq:-27}
\end{equation}
Then
\begin{equation}
\begin{aligned}\partial_{s^{(l)}}f(s) & =\mathbb{E}\Big[F(S)\Big(\sum_{k=0}^{1}\frac{1}{(\Delta t_{1})^{k/2}}\Theta_{l}^{(k)}(s,Z_{1},\sqrt{\Delta t_{1}}Z_{1})\Big)\Big]\end{aligned}
,\label{eq:-19-1}
\end{equation}
where
\begin{align}
\Theta_{l}^{(0)}(s,z,x) & =-\sum_{p=1}^{n}\partial_{x^{(p)}}J_{pl}(s,x),\label{eq:-20}\\
\Theta_{l}^{(1)}(s,z,x) & =\sum_{p=1}^{n}z^{(p)}J_{pl}(s,x),\label{eq:-21}
\end{align}
with $s,z,x\in\mathbb{R}^{1\times n}$. Moreover,
\begin{equation}
\begin{aligned}\partial_{s^{(l)}s^{(m)}}^{2}f(s) & =\mathbb{E}\Big[F(S)\Big(\sum_{k=0}^{2}\frac{1}{(\Delta t_{1})^{k/2}}\Lambda_{lm}^{(k)}(s,Z_{1},\sqrt{\Delta t_{1}}Z_{1})\Big)\Big]\end{aligned}
,\label{eq:-11}
\end{equation}
where 
\begin{align}
\Lambda_{lm}^{(0)}(s,z,x) & =\sum_{p,r}\big(\partial_{x^{(r)}}J_{rm}\partial_{x^{(p)}}J_{pl}\big)(s,x)\nonumber \\
 & \quad+\sum_{p,r}\big(J_{rm}\partial_{x^{(r)}x^{(p)}}^{2}J_{pl}\big)(s,x)-\sum_{p}\partial_{s^{(m)}x^{(p)}}^{2}J_{pl}(s,x),\label{eq:-5}\\
\Lambda_{lm}^{(1)}(s,z,x) & =\sum_{p}z^{(p)}\partial_{s^{(m)}}J_{pl}(s,x)-\sum_{p,r}z^{(p)}\big(J_{rm}\partial_{x^{(r)}}J_{pl}\big)(s,x)\nonumber \\
 & \quad-\sum_{p,r}z^{(p)}\big(\partial_{x^{(r)}}J_{rm}\cdot J_{pl}\big)(s,x)-\sum_{p,r}z^{(r)}\big(J_{rm}\cdot\partial_{x^{(p)}}J_{pl}\big)(s,x),\label{eq:-12}\\
\Lambda_{lm}^{(2)}(s,z,x) & =\sum_{p,r}z^{(r)}z^{(p)}(J_{rm}J_{pl}-\delta_{rp}J_{rm}J_{pl})(s,x),\label{eq:-13}
\end{align}
with $s,z,x\in\mathbb{R}^{1\times n}$ and $\delta_{rp}$ being the
Dirac notation. 
\end{thm}

\begin{rem}
We should point out that the definitions of $\Theta^{(k)}$ and $\Lambda^{(k)}$
treat $z$ and $x$ as independent variables. This treatment is crusial
for the variance reduction formula in Section \ref{sec:}.
\end{rem}

\begin{rem}
For a basket option with basket size $n$, the finite difference approximation
of the Delta vector requires $2nM$ simulations, and that of the Gamma
matrix requires $[3n+4\cdot(n^{2}-n)/2]=(2n^{2}+n)M$ simulations\footnote{Each diagonal gamma requires $3M$ simulations, and each cross gamma
requires $4M$ simulations.}, where $M$ is the number of sample paths. In contrast, the path
weighting formula of the Delta vector and Gamma matrix only requires
$M$ simulations, plus a few matrix operator at $t_{1}$ only, which
is usually less expensive than applying the correlation matrix to
the driving white noise. The computation cost reduction of the path
weighting sensitivities is clearly an advantage. Moreover, the path
weighting sensitivities has the advantage of being unbiased. However,
as we shall see at the beginning of Section \ref{sec:}, the path
weighting sensitivities also suffers from large variance. This is
the main bottleneck preventing it from being useful in practice.
\end{rem}

Equations (\ref{eq:-20})--(\ref{eq:-21}) and (\ref{eq:-5})--(\ref{eq:-13})
can be written more concisely using matrix notation. This is also
preferrable for programming languages which parallelise matrix operations
. Denote the divergence
\[
\mathrm{div}_{x}(J)=\Big(\sum_{p}\partial_{x^{(p)}}J_{p1},\dots,\sum_{p}\partial_{x^{(p)}}J_{pn}\Big).
\]
The derivative matrices of $\mathrm{div}_{x}(J)$ is
\[
\begin{aligned}\big[\partial_{x}\big(\mathrm{div}_{x}(J)\big)\big]_{lr} & =\partial_{x^{(r)}}\Big(\sum_{p}\partial_{x^{(p)}}J_{pl}\Big)=\sum_{p}\partial_{x^{(r)}x^{(p)}}^{2}J_{pl},\\
\big[\partial_{x}\big(\mathrm{div}_{x}(J)\big)\big]_{lm} & =\partial_{s^{(m)}}\Big(\sum_{p}\partial_{x^{(p)}}J_{pl}\Big)=\sum_{p}\partial_{s^{(m)}x^{(p)}}^{2}J_{pl}.
\end{aligned}
\]
Moreover,
\begin{align}
\Theta^{(0)}(s,z,x) & =-\mathrm{div}_{x}(J)(s,x),\label{eq:-20-1}\\
\Theta_{l}^{(1)}(s,z,x) & =(z^{\text{T}}J)(s,x),\label{eq:-21-1}
\end{align}
and
\begin{align}
\Lambda^{(0)}(s,z,x) & =\big[\mathrm{div}_{x}(J)^{\text{T}}\mathrm{div}_{x}(J)+\partial_{x}\big(\mathrm{div}_{x}(J)\big)\cdot J-\partial_{s}\big(\mathrm{div}_{x}(J)\big)\big](s,x),\label{eq:-15}\\
\Lambda^{(1)}(s,z,x) & =\big[\partial_{s}\big(z^{\text{T}}J\big)-\partial_{x}\big(z^{\text{T}}J\big)\cdot J-\big(z^{\text{T}}J\big)^{\text{T}}\mathrm{div}_{x}(J)-\mathrm{div}_{x}(J)^{\text{T}}\big(z^{\text{T}}J\big)\big](s,x),\label{eq:-18}\\
\Lambda^{(2)}(s,z,x) & =\big[(z^{\text{T}}J)^{\text{T}}(z^{\text{T}}J)-J^{\text{T}}J\big](s,x).\label{eq:-17}
\end{align}

\begin{rem}
Note that (\ref{eq:-11}) and the symmetricity of the matrix $\partial^{2}f$
implies that $\Lambda^{(0)}$, $\Lambda^{(1)}$, and $\Lambda^{(2)}$
are symmetric matrices. While the matrix $\Lambda^{(2)}$ is clearly
symmetric, it is not immediate for the matrices $\Lambda^{(0)}$ and
$\Lambda^{(1)}$. Lemma \ref{lem:-2} below confirms that $\Lambda^{(0)}$
and $\Lambda^{(1)}$ are both symmetric. In Appendix \ref{subsec:-3},
we give an algebraic proof to Lemma \ref{lem:-2}.
\end{rem}

\begin{lem}
\label{lem:-2}(i) 
\[
\sum_{r}\partial_{x^{(r)}}J_{pm}\cdot J_{rl}+\partial_{s^{(m)}}J_{pl}
\]
is symmetric with respect to $(m,l)$.

(ii) 
\[
\sum_{p,r}J_{rm}\partial_{x^{(r)}x^{(p)}}^{2}J_{rl}-\sum_{p}\partial_{s^{(m)}x^{(p)}}^{2}J_{pl}
\]
is symmetric with respect to $(m,l)$.

(iii) The matrices $\Lambda^{(0)}$ and $\Lambda^{(1)}$ are symmetric.
\end{lem}

\subsection{\label{subsec:}Correlated Diagonal Dynamics: An Application}

In this subsection, we apply the result in Section \ref{subsec:-1}
to a type of dynamics which is widely used in the industry. 
\begin{defn}
The dynamics (\ref{eq:-14}) is called \emph{correlated diagonal}
if there exists a functional
\[
L(s,w)=(L^{(1)}(s^{(1)},w^{(n)}),\dots,L^{(n)}(s^{(n)},w^{(n)}))
\]
and a correlation matrix $\rho\rho^{\text{T}}$ such that
\begin{equation}
H_{0}^{(j)}(s,x)=L^{(j)}\Big(s_{j},\sum_{r=1}^{n}\rho_{jr}x^{(r)}\Big),\;1\le j\le n.\label{eq:-16}
\end{equation}
\end{defn}

\begin{rem}
For correlated diagonal dynamics, $S^{(i)}$ and $S^{(j)}$ interacts
with each other only via the correlation of the driving Brownian motions. 
\end{rem}

Denote
\[
R(s,w)=\big(R^{(1)}(s^{(1)},w^{(1)}),\dots,R^{(n)}(s^{(n)},w^{(n)})\big)=\Big(\frac{\partial_{s^{(1)}}L^{(1)}}{\partial_{x^{(1)}}L^{(1)}}(s^{(1)},w^{(1)}),\dots,\frac{\partial_{s^{(n)}}L^{(n)}}{\partial_{w^{(n)}}L^{(n)}}(s^{(n)},w^{(n)})\Big),
\]
and
\[
\begin{aligned}\partial_{w}\odot R & =\big(\partial_{w^{(1)}}R^{(1)},\dots,\partial_{w^{(n)}}R^{(n)}\big),\\
\partial_{s}\odot R & =\big(\partial_{s^{(1)}}R^{(1)},\dots,\partial_{s^{(n)}}R^{(n)}\big),\\
\partial_{w}\odot\partial_{w}\odot R & =\big(\partial_{w^{(1)}w^{(1)}}^{2}R^{(1)},\dots,\partial_{w^{(n)}w^{(n)}}^{2}R^{(n)}\big),\\
\partial_{s}\odot\partial_{w}\odot R & =\big(\partial_{s^{(1)}w^{(1)}}^{2}R^{(1)},\dots,\partial_{s^{(n)}w^{(n)}}^{2}R^{(n)}\big).
\end{aligned}
\]
It can be easily shown that
\begin{align}
\Theta^{(0)}(s,z,x) & =-(\partial_{w}\odot R)(s,\rho x),\label{eq:-22}\\
\Theta^{(1)}(s,z,x) & =\big[(z^{\text{T}}\rho^{-1})\mathrm{diag}(R)\big](s,\rho x),\label{eq:-23}
\end{align}
and
\begin{align}
\Lambda^{(0)}(s,z,x) & =\big[(\partial_{w}\odot R)^{\text{T}}(\partial_{w}\odot R)+\mathrm{diag}(\partial_{w}\odot\partial_{w}\odot R)\cdot\mathrm{diag}(R)-\mathrm{diag}(\partial_{s}\odot\partial_{w}\odot R)\big](s,\rho x),\label{eq:-24}\\
\Lambda^{(1)}(s,z,x) & =\big[\mathrm{diag}(z^{\text{T}}\rho^{-1})\mathrm{diag}(\partial_{s}\odot R)-\mathrm{diag}(z^{\text{T}}\rho^{-1})\mathrm{diag}(\partial_{x}\odot R)\mathrm{diag}(R)\nonumber \\
 & \quad-\mathrm{diag}(R)(z^{\text{T}}\rho^{-1})^{\text{T}}(\partial_{x}\odot R)-(\partial_{x}\odot R)^{\text{T}}(z^{\text{T}}\rho^{-1})\mathrm{diag}(R)\big](s,\rho x),\label{eq:-25}\\
\Lambda^{(2)}(s,z,x) & =\big[\mathrm{diag}(R)(z^{\text{T}}\rho^{-1})^{\text{T}}(z^{\text{T}}\rho^{-1})\mathrm{diag}(R)-\mathrm{diag}(R)(\rho^{-1})^{\text{T}}(\rho^{-1})\mathrm{diag}(R)\big](s,\rho x).\label{eq:-26}
\end{align}

\begin{example}
The local volatility model is an example of correlated diagonal dynamics.
Assuming zero risk-free rate\footnote{This assumption is not restrictive, since we may model the discounted
stock prices.}, it is given as the solution to the SDE
\[
\begin{aligned} & dS_{t}^{(j)}=\sigma^{(j)}(t,S_{t}^{(j)})dW_{t}^{(j)},\\
 & d\langle W^{(j)},W^{(k)}\rangle_{t}=\rho_{jk}dt,\;1\le j,k\le n,
\end{aligned}
\]
where $\sigma^{(j)}(t_{i},s^{(j)})$, $1\le j\le n$ are bounded functions
and twice differentiable in $s^{(j)}$. For the discredtisation
\[
S_{t_{i+1}}^{(j)}=S_{t_{i}}^{(j)}\cdot\exp\Big(-\frac{\sigma_{i}^{(j)}(S_{t_{i}}^{(j)})^{2}\Delta t_{i+1}}{2}+\sigma_{i}^{(j)}(S_{t_{i}}^{(k)})\sqrt{\Delta t_{i+1}}\sum_{k=1}^{n}\rho_{jk}Z_{i+1}^{(k)}\Big),
\]
where $\sigma_{i}^{(j)}(s^{(j)})=\sigma^{(j)}(t_{i},s^{(j)})$, by
simple calculation, the function $R(s,w)$ and its derivatives are
given by
\begin{equation}
\begin{aligned}R(s,w) & =c_{2}(s)+w^{\text{T}}\mathrm{diag}(c_{1}(s)),\\
(\partial_{w}\odot R)(s,w) & =c_{1}(s),\\
(\partial_{s}\odot R)(s,w) & =(\partial_{s}\odot c_{2})(s)+w^{\text{T}}\mathrm{diag}\big((\partial_{s}\odot c_{1})(s)\big),\\
(\partial_{s}\odot\partial_{w}\odot R)(s,w) & =(\partial_{s}\odot c_{1})(s),\\
(\partial_{w}\odot\partial_{w}\odot R)(s,w) & =0,
\end{aligned}
\label{eq:-55}
\end{equation}
where
\begin{equation}
\begin{aligned}c_{1}(s) & =\Big(\frac{\partial_{s^{(1)}}\sigma_{0}^{(1)}(s^{(1)})}{\sigma_{0}^{(1)}(s_{1})},\dots,\frac{\partial_{s^{(n)}}\sigma_{0}^{(n)}(s^{(n)})}{\sigma_{0}^{(n)}(s_{n})}\Big),\\
c_{2}(s) & =\Big(\frac{1}{s^{(1)}\sigma_{0}^{(1)}(s^{(1)})}-\Delta t_{1}\cdot\partial_{s^{(1)}}\sigma_{0}^{(1)}(s^{(1)}),\dots,\frac{1}{s^{(n)}\sigma_{0}^{(n)}(s^{(n)})}-\Delta t_{1}\cdot\partial_{s^{(n)}}\sigma_{0}^{(n)}(s^{(n)})\Big),\\
(\partial_{s}\odot c_{1})(s) & =\Big(\frac{\partial_{s^{(1)}}^{2}\sigma_{0}^{(1)}(s^{(1)})}{\sigma_{0}^{(1)}(s^{(1)})}-c_{1}^{(1)}(s^{(1)})^{2},\dots,\frac{\partial_{s^{(n)}}^{2}\sigma_{0}^{(n)}(s^{(n)})}{\sigma_{0}^{(n)}(s^{(n)})}-c_{1}^{(n)}(s^{(n)})^{2}\Big),\\
(\partial_{s}\odot c_{2})(s) & =\Big(-\frac{1+s^{(1)}c_{1}^{(1)}(s^{(1)})}{(s^{(1)})^{2}\sigma_{0}^{(1)}(s^{(1)})}-\Delta t_{1}\cdot\partial_{s^{(1)}}^{2}\sigma_{0}^{(1)}(s^{(1)}),\dots,-\frac{1+s^{(n)}c_{1}^{(n)}(s^{(n)})}{(s^{(n)})^{2}\sigma_{0}^{(n)}(s^{(n)})}-\Delta t_{1}\cdot\partial_{s^{(n)}}^{2}\sigma_{0}^{(n)}(s^{(n)})\Big).
\end{aligned}
\label{eq:-56}
\end{equation}
\end{example}

\section{\label{sec:}Variance Reduction for Small Time Step}

The formulae (\ref{eq:-19-1}) and (\ref{eq:-11}) usually suffer
from large variance compared to other methods such as finte difference
approximation. The large variances of (\ref{eq:-19-1}) and (\ref{eq:-11})
are mainly due to the appearance of factors $(\Delta t_{1})^{-1/2}$
and $(\Delta t_{1})^{-1}$ (though, as we shall see in Section \ref{sec:-2},
close-to-singular correlation matrix is another cause). The following
example shows that the variances of the estimator in formulae (\ref{eq:-19-1})
and (\ref{eq:-11}) are of orders $\mathrm{O}\big((\Delta t_{1})^{-1}\big)$
and $\mathrm{O}\big((\Delta t_{1})^{-2}\big)$ respectively as $\Delta t_{1}\to0$.
\begin{example}
\label{exa:}Suppose that the stock dynamics is given by the following
two-step scheme for the one-dimensional Black--Scholes model
\[
\begin{aligned}S_{t_{1}} & =s\cdot e^{-\frac{\sigma^{2}\Delta t_{1}}{2}+\sigma\sqrt{\Delta t_{1}}Z_{1}},\\
S_{T} & =S_{t_{1}}\cdot e^{-\frac{\sigma^{2}(T-\Delta t_{1})}{2}+\sigma\sqrt{T-\Delta t_{1}}Z_{2}}.
\end{aligned}
\]
Then
\[
\begin{aligned}H_{0}(s,x) & =s\cdot e^{-\frac{\sigma^{2}\Delta t_{1}}{2}+\sigma x},\\
J(s,x) & =\frac{1}{s\sigma}.
\end{aligned}
\]
Formula (\ref{eq:-19-1}) gives
\begin{equation}
\partial_{s}\mathbb{E}(F(S_{T}))=\frac{1}{s\sigma\sqrt{\Delta t_{1}}}\mathbb{E}\Big[F\big(se^{-\frac{\sigma^{2}T}{2}+\sigma\sqrt{\Delta t_{1}}Z_{1}+\sigma\sqrt{T-\Delta t_{1}}Z_{2}}\big)Z_{1}\Big].\label{eq:-35}
\end{equation}
Consider the simple case when $F(S_{T})=1$ is a constant. The random
variable
\[
\xi=F\big(se^{-\frac{\sigma^{2}T}{2}+\sigma\sqrt{\Delta t_{1}}Z_{1}+\sigma\sqrt{T-\Delta t_{1}}Z_{2}}\big)Z_{1}=\frac{Z_{1}}{s\sigma\sqrt{\Delta t_{1}}}
\]
has variance
\[
\mathrm{Var}(\xi)=\mathrm{O}\big((\Delta t_{1})^{-1}\big)
\]
as $\Delta t_{1}\to0$. Therefore, the variance of the MC estimator
for (\ref{eq:-35}) is of order $\text{O}((\Delta t_{1})^{-1}M^{-1/2})$,
where $M$ is the number of sample path.
\end{example}

In the rest of this section, we shall derive an adjusted formula to
reduce the variance of (\ref{eq:-19-1}) and (\ref{eq:-11}). 
\begin{defn}
We denote by $\hat{Z}$ the vector obtained by reflecting the first
element of $Z$, i.e.
\begin{equation}
\hat{Z}=(-Z_{1},Z_{2},\dots,Z_{N}),\label{eq:-19}
\end{equation}
and by $\bar{Z}$ the vector obtained by replacing the first element
of $Z$ by zero, i.e.
\begin{equation}
\bar{Z}=(0,Z_{2},\dots,Z_{N}).\label{eq:-57}
\end{equation}
Moreover, the process $\hat{S}_{t}$ and $\bar{S}_{t}$ are defined
to be generated by replacing $Z$ with $\hat{Z}$ and $\bar{Z}$ in
(\ref{eq:-1}) respectively, i.e. 
\begin{align}
\hat{S}_{t_{i+1}} & =H_{i}(\hat{S}_{t_{i}},\sqrt{\Delta t_{i+1}}\hat{Z}_{i+1}),\;\hat{S}_{t_{0}}=s,\label{eq:-58}\\
\bar{S}_{t_{i+1}} & =H_{i}(\bar{S}_{t_{i}},\sqrt{\Delta t_{i+1}}\bar{Z}_{i+1}),\;\bar{S}_{t_{0}}=s.\label{eq:-59}
\end{align}
\end{defn}

We may now formulate the variance reduction adjustment result. As
we shall see from the derivation of Theorem \ref{thm:}, for general
diffusion models such as the local volatility model, a crucial treatment
in the adjustment formula is that $\Theta^{(k)}$ and $\Lambda^{(k)}$
distinguish different appearances of the white noise $Z_{1}$.
\begin{thm}
\label{thm:}Formulae (\ref{eq:-19-1}) and (\ref{eq:-11}) admit
the equivalently adjusted forms 
\begin{equation}
\begin{aligned}\partial_{s^{(l)}}f(s) & =\mathbb{E}\Big[F(S)\Theta_{l}^{(0)}(s,Z_{1},\sqrt{\Delta t_{1}}Z_{1})\\
 & \quad+\frac{1}{2\sqrt{\Delta t_{1}}}\Big(F(S)\Theta_{l}^{(1)}(s,Z_{1},\sqrt{\Delta t_{1}}Z_{1})-F(\hat{S})\Theta_{l}^{(1)}(s,Z_{1},\sqrt{\Delta t_{1}}\hat{Z}_{1})\Big)\Big].
\end{aligned}
\label{eq:-28}
\end{equation}
and
\begin{equation}
\begin{aligned}\partial_{s^{(l)}s^{(m)}}^{2}f(s) & =\mathbb{E}\Big[F(S)\frac{1}{(\Delta t_{1})^{k/2}}\Lambda_{lm}^{(0)}(s,Z_{1},\sqrt{\Delta t_{1}}Z_{1})\\
 & \quad+\frac{1}{2\sqrt{\Delta t_{1}}}\Big(F(S)\Lambda_{lm}^{(1)}(s,Z_{1},\sqrt{\Delta t_{1}}Z_{1})-F(\hat{S})\Lambda_{lm}^{(1)}(s,Z_{1},\sqrt{\Delta t_{1}}\hat{Z}_{1})\Big)\\
 & \quad+\frac{1}{2\Delta t_{1}}\Big(F(S)\Lambda_{lm}^{(2)}(s,Z_{1},\sqrt{\Delta t_{1}}Z_{1})-2F(\bar{S})\Lambda_{lm}^{(2)}(s,Z_{1},\sqrt{\Delta t_{1}}\bar{Z}_{1})+F(\hat{S})\Lambda_{lm}^{(2)}(s,Z_{1},\sqrt{\Delta t_{1}}\hat{Z}_{1})\Big)\Big]
\end{aligned}
\label{eq:-29}
\end{equation}
\end{thm}

\begin{proof}
By the symmetricity of $Z_{1}$,
\[
\begin{aligned} & \mathbb{E}\Big(F(S)\Theta_{l}^{(1)}(s,Z_{1},\sqrt{\Delta t_{1}}Z_{1})\Big)\\
 & =\mathbb{E}\Big(F(\hat{S})\Theta_{l}^{(1)}(s,\hat{Z}_{1},\sqrt{\Delta t_{1}}\hat{Z}_{1})\Big)\\
 & =\frac{1}{2}\mathbb{E}\Big(F(S)\Theta_{l}^{(1)}(s,Z_{1},\sqrt{\Delta t_{1}}Z_{1})+F(\hat{S})\Theta_{l}^{(1)}(s,\hat{Z}_{1},\sqrt{\Delta t_{1}}\hat{Z}_{1})\Big).
\end{aligned}
\]
Note that $\Theta_{l}^{(1)}(s,\hat{Z}_{1},\sqrt{\Delta t_{1}}\hat{Z}_{1})=-\Theta_{l}^{(1)}(s,Z_{1},\sqrt{\Delta t_{1}}\hat{Z}_{1})$.
The above can be written as
\[
\mathbb{E}\Big(F(S)\Theta_{l}^{(1)}(s,Z_{1},\sqrt{\Delta t_{1}}Z_{1})\Big)=\frac{1}{2}\mathbb{E}\Big(F(S)\Theta_{l}^{(1)}(s,Z_{1},\sqrt{\Delta t_{1}}Z_{1})-F(\hat{S})\Theta_{l}^{(1)}(s,Z_{1},\sqrt{\Delta t_{1}}\hat{Z}_{1})\Big),
\]
which, together with formula (\ref{eq:-19-1}), proves (\ref{eq:-28})

Similarly, we have
\begin{equation}
\mathbb{E}\Big(F(S)\Lambda_{lm}^{(1)}(s,Z_{1},\sqrt{\Delta t_{1}}Z_{1})\Big)=\frac{1}{2}\mathbb{E}\Big(F(S)\Lambda_{lm}^{(1)}(s,Z_{1},\sqrt{\Delta t_{1}}Z_{1})-F(\hat{S})\Lambda_{lm}^{(1)}(s,Z_{1},\sqrt{\Delta t_{1}}\hat{Z}_{1})\Big).\label{eq:-32}
\end{equation}
Moreover, note that $\bar{Z}$ and $\bar{S}$ are independent of $Z_{1}$.
Therefore,
\begin{equation}
\begin{aligned}\mathbb{E}\Big(F(\bar{S})\Lambda^{(2)}(s,Z_{1},\sqrt{\Delta t_{1}}\bar{Z}_{1})\Big) & =\mathbb{E}\Big(F(\bar{S})\big[J(s,0)^{\text{T}}Z_{1}Z_{1}^{\text{T}}J(s,0)-J(0,s)^{\text{T}}J(0,s)\big]\Big)\\
 & =\mathbb{E}\big(F(\bar{S})\big)\cdot J(s,0)^{\text{T}}\cdot\big[\mathbb{E}\big(Z_{1}Z_{1}^{\text{T}}\big)-I\big]\cdot J(0,s)=0.
\end{aligned}
\label{eq:-30}
\end{equation}
By symmetrcity of $Z_{1}$ and the fact that $\Lambda_{lm}^{(2)}(s,z,x)$
is even in $z$, we have
\begin{equation}
\begin{aligned}\mathbb{E}\Big(F(S)\Lambda_{lm}^{(2)}(s,Z_{1},\sqrt{\Delta t_{1}}Z_{1})\Big) & =\mathbb{E}\Big(F(\hat{S})\Lambda_{lm}^{(2)}(s,\hat{Z}_{1},\sqrt{\Delta t_{1}}\hat{Z}_{1})\Big)=\mathbb{E}\Big(F(\hat{S})\Lambda_{lm}^{(2)}(s,Z_{1},\sqrt{\Delta t_{1}}\hat{Z}_{1})\Big).\end{aligned}
\label{eq:-31}
\end{equation}
Equations (\ref{eq:-30}) and (\ref{eq:-31}) imply
\begin{equation}
\begin{aligned}\mathbb{E}\Big(F(S)\Lambda_{lm}^{(2)}(s,Z_{1},\sqrt{\Delta t_{1}}Z_{1})\Big) & =\frac{1}{2}\mathbb{E}\Big(F(S)\Lambda_{lm}^{(2)}(s,Z_{1},\sqrt{\Delta t_{1}}Z_{1})\\
 & \quad-2F(\bar{S})\Lambda_{lm}^{(2)}(s,Z_{1},\sqrt{\Delta t_{1}}\bar{Z}_{1})+F(\hat{S})\Lambda_{lm}^{(2)}(s,Z_{1},\sqrt{\Delta t_{1}}\hat{Z}_{1})\Big).
\end{aligned}
\label{eq:-33}
\end{equation}
Equation (\ref{eq:-29}) follows immediately from (\ref{eq:-11}),
(\ref{eq:-32}) and (\ref{eq:-33}).
\end{proof}
\begin{example}
\label{exa:-1}Consider an autocallable payoff $F(S)$ on a single
asset with the following settings
\begin{itemize}
\item maturity: 1.0
\item knock-out times: $\Big(\frac{90}{360},\frac{180}{360},\frac{270}{360},\frac{360}{360}\Big)$;
\item knock-out coupons: $(0.1,0.1,0.1,0.1)$;
\item knock-out barriers: $(0.95,0.9,0.85,0.8)$;
\item knock-in barrier: $0.75$;
\item knock-in put strike: 1.
\end{itemize}
On each knock-out time $t$, if the stock performance $S_{t}/S_{t_{0}}$
is greater than or equal to the corresponding knock-out barrier, the
option is terminated and pays the corresponding knock-out coupon to
the option holder. If the option does not knock out on any knock-out
time, then it pays a European down-and-in put to the option holder.
We use the local volatility dynamics with the local volatility surface
given by the following parameterisation
\begin{itemize}
\item tenors: $\Big(\frac{5}{360},\frac{30}{360},\frac{90}{360},\frac{180}{360},\frac{360}{360},\frac{720}{360}\Big)$;
\item at-the-money volatilities: $(0.3,0.35,0.3,0.32,0.3,0.3)$;
\item skews: $(-0.05,-0.04,0.03,0.10,0.10,0.10)$;
\item kurtosises: $(0.02,0.02,0.05,0.10,0.10,0.10)$.
\end{itemize}
For SDE discretisation, we use $\Delta t_{1}=\frac{1}{360}$ for the
firt time step and $\Delta t_{i}=\frac{7}{360}$ for the rest. In
Figure \ref{fig:-1}, ``Delta{[}FD{]}'' refers to the delta computed
using finite difference approximaytion, ``Delta{[}PW\_raw{]}'' refers
to that computed using formulae (\ref{eq:-19-1}) and (\ref{eq:-11}),
and ``Delta{[}PW\_adj{]}'' refers to that computed using the variance
reduction adjustment in Theorem \ref{thm:}. The two grey curves represent
the 99\% confidence interval for the finite difference approximation.
The legends in the other graph have similar meanings.

Numerical test results for this example are generated by the test
case \texttt{test\_unadjusted\_pw\_variance} in \texttt{pwsen}.

\begin{figure}[H]
\centering{}%
\begin{tabular}{cc}
\includegraphics[scale=0.5]{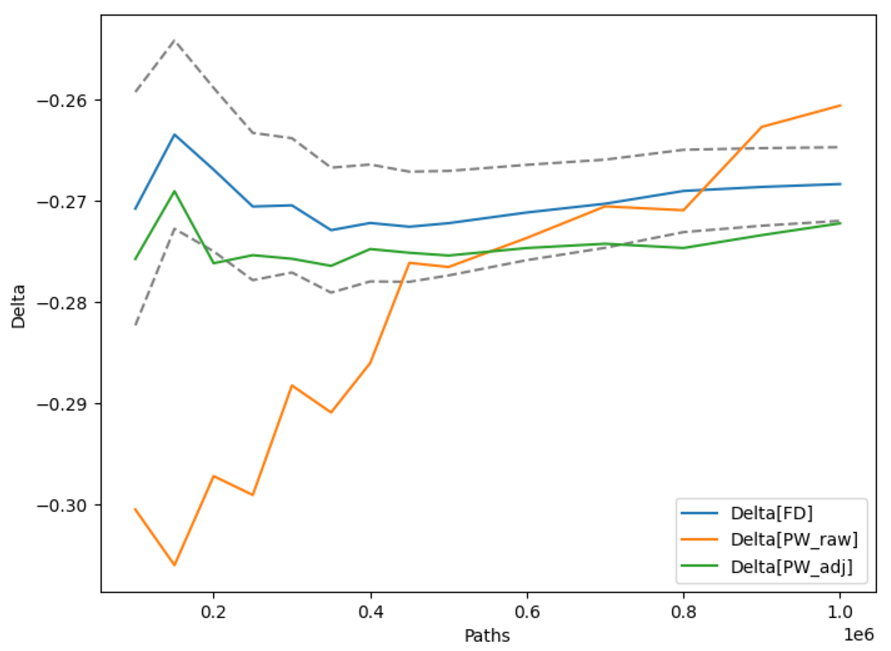} & \includegraphics[scale=0.5]{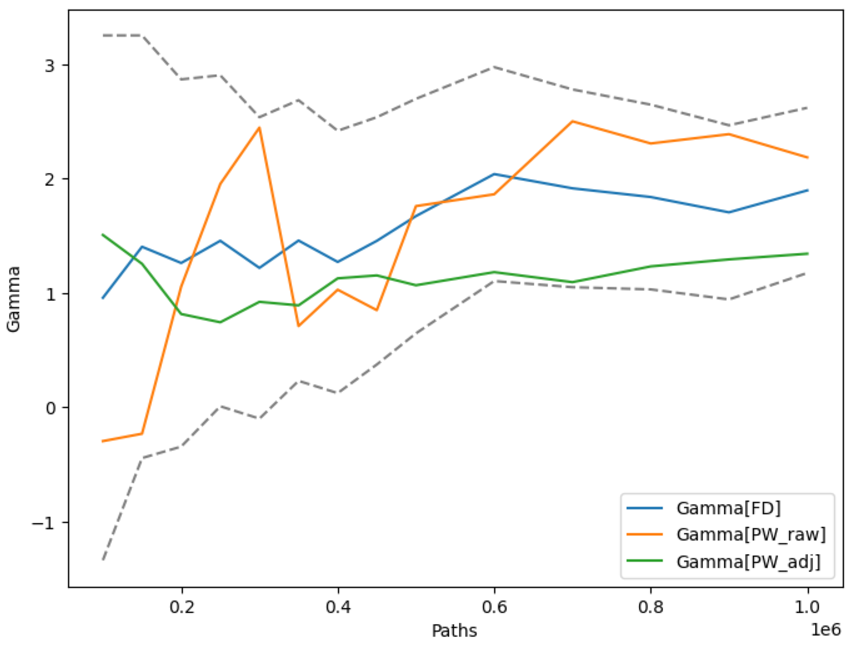}\tabularnewline
\end{tabular}\caption{Convergence Path}
\label{fig:-1}
\end{figure}
\end{example}

\begin{example}
To verify Theorem \ref{thm:} gives a correct estimator, we compare
the Delta and Gamma spot ladders computed using finite difference
approximation and Theorem \ref{thm:} (using 200k paths) for an autocallable
with the following set-up
\end{example}

\begin{itemize}
\item maturity: $1.0$;
\item number of assets: $2$;
\item knock-out times: $\Big(\frac{2}{360},\frac{180}{360},\frac{270}{360},\frac{360}{360}\Big)$;
\item knock-out coupons: $(0.1,0.1,0.1,0.1)$;
\item knock-out barriers: $(1.0,0.95,0.9,0.85)$;
\item knock-in barrier: $0.75$;
\item knock-in put strike: 1.
\end{itemize}
The parameterisation of volatility surfaces are using at-the-money
volatilities around $0.2$, skews varing from $-0.05$ to $0.05$,
and kurtosises varying from $0.02$ to $0.15$. Results shown in Firgure
\ref{fig:-6}, Figure \ref{fig:-7}, and Figure \ref{fig:-8} are
generated by the test case \texttt{test\_spot\_ladder} in \texttt{pwsen}.

\begin{figure}[H]
\begin{centering}
\includegraphics[scale=0.7]{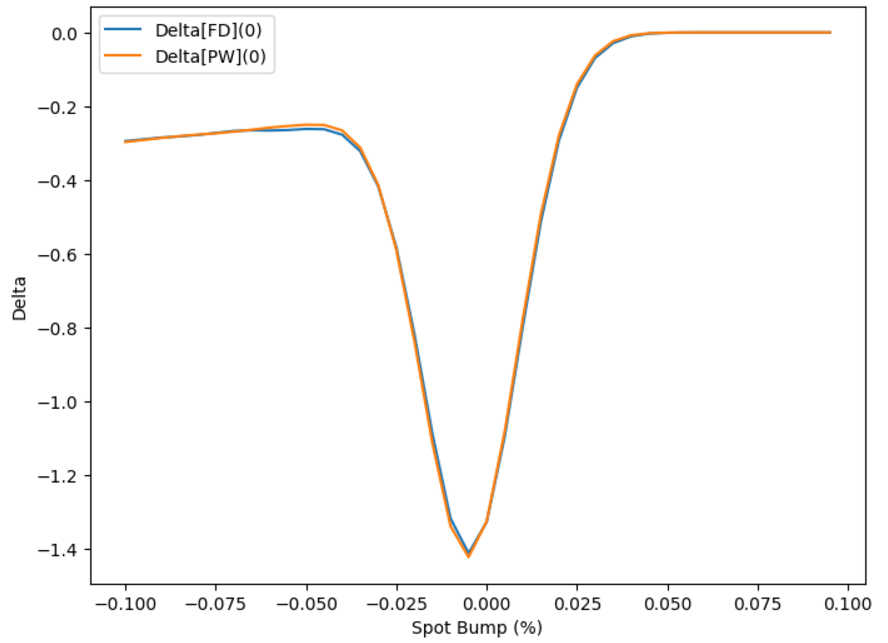}\caption{Delta Ladder}
\label{fig:-6}
\par\end{centering}
\end{figure}

\begin{figure}[H]
\centering{}\includegraphics[scale=0.7]{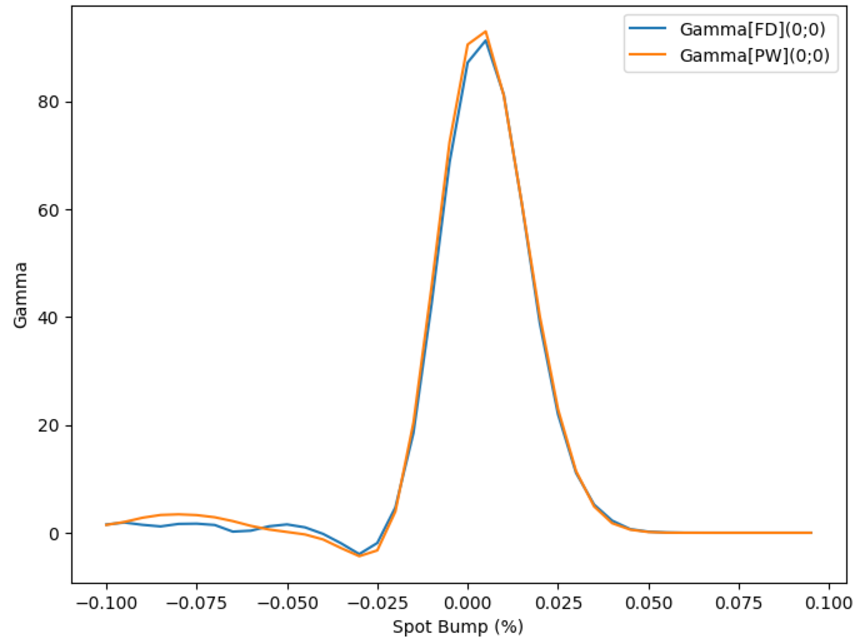}\caption{Diagonal Gamma Ladder}
\label{fig:-7}
\end{figure}
\begin{figure}[H]
\centering{}\includegraphics[scale=0.7]{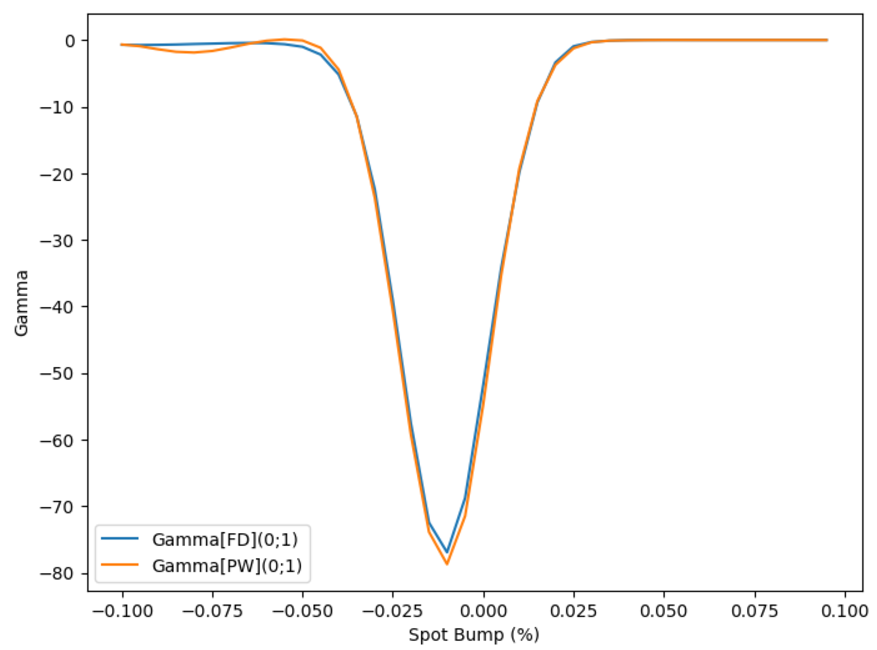}\caption{Cross Gamma Ladder}
\label{fig:-8}
\end{figure}

To derive asymptotic order of variance of the estimators in Theorem
\ref{thm:}, we will need the following regularity assumption. 
\begin{assumption}
\label{assu:}The functionals $H_{i}$ have continuous second order
derivatives with at most polynomial growth at infinity. And for any
$s$, there exists a polynomial $p_{s}(x)$ such that 
\[
\Vert J_{0}(s,x)\Vert_{l^{\infty}}=\Vert(\partial_{x}H_{0})^{-1}(\partial_{s}H_{0})\Vert_{l^{\infty}}\le|p_{s}(x)|,
\]
where $\Vert\cdot\Vert_{l^{\infty}}$ is the $l^{\infty}$-norm $\Vert A\Vert_{l^{\infty}}=\max_{i,j}|A_{ij}|$. 
\end{assumption}

\begin{rem}
Assumption \ref{assu:} is usually satisfied and can be easily verified.
The local volatility model is such an example.
\end{rem}

\begin{thm}
\label{thm:-1}If $F(S)$ has generalised derivatives with at most
polynomial growth a.e., then the random variables 
\begin{equation}
\xi=\frac{1}{\sqrt{\Delta t_{1}}}\Big(F(S)\Theta^{(1)}(s,Z_{1},\sqrt{\Delta t_{1}}Z_{1})-F(\hat{S})\Theta^{(1)}(s,Z_{1},\sqrt{\Delta t_{1}}\hat{Z}_{1})\Big),\label{eq:-43}
\end{equation}
\begin{equation}
\eta_{1}=\frac{1}{\sqrt{\Delta t_{1}}}\Big(F(S)\Lambda^{(1)}(s,Z_{1},\sqrt{\Delta t_{1}}Z_{1})-F(\hat{S})\Lambda^{(1)}(s,Z_{1},\sqrt{\Delta t_{1}}\hat{Z}_{1})\Big),\label{eq:-44}
\end{equation}
have variance of order
\begin{align}
\mathrm{Var}(\xi) & =\mathrm{O}(1),\;(\Delta t_{1}\to0),\label{eq:-36}\\
\mathrm{Var}(\eta_{1}) & =\mathrm{O}(1),\;(\Delta t_{1}\to0).\label{eq:-37}
\end{align}
If, in addition, $F(S)$ has second order generalised derivatives
with at most polynomial growth a.e., then the random variable
\begin{equation}
\eta_{2}=\frac{1}{\Delta t_{1}}\Big(F(S)\Lambda^{(2)}(s,Z_{1},\sqrt{\Delta t_{1}}Z_{1})-2F(\bar{S})\Lambda^{(2)}(s,Z_{1},\sqrt{\Delta t_{1}}\bar{Z}_{1})+F(\hat{S})\Lambda^{(2)}(s,Z_{1},\sqrt{\Delta t_{1}}\hat{Z}_{1})\Big)\label{eq:-45}
\end{equation}
has variance of order
\begin{equation}
\mathrm{Var}(\eta_{1})=\mathrm{O}(1),\;(\Delta t_{1}\to0),\label{eq:-38}
\end{equation}
\end{thm}

\begin{proof}
Clearly, we may write
\begin{equation}
F(S_{t_{1}},\dots,S_{t_{N}})=\tilde{F}(\sqrt{\Delta t_{1}}Z_{1},\dots,\sqrt{\Delta t_{N}}Z_{N}).\label{eq:-39}
\end{equation}
Since $\partial F$ has at most polynomial growth, the function
\[
G(x_{1},\dots,x_{N})=\tilde{F}(x_{1},\dots,x_{N})Z_{1}^{\text{T}}J(s,x_{1}),
\]
has generalised derivatives with at most exponential growth a.e..
Hence, there exists a constant $C>0$ such that 
\[
|\partial_{1}G(\lambda Z_{1},\sqrt{\Delta t_{2}}Z_{2},\dots,\sqrt{\Delta t_{N}}Z_{N})|\le e^{C\sqrt{t_{N}}|Z|},\;|\lambda|\le\sqrt{\Delta t_{1}}.
\]
Moreover,
\begin{align*}
|\xi| & =\frac{1}{\sqrt{\Delta t_{1}}}|G(\sqrt{\Delta t_{1}}Z_{1},\dots,\sqrt{\Delta t_{N}}Z_{N})-G(-\sqrt{\Delta t_{1}}Z_{1},\dots,\sqrt{\Delta t_{N}}Z_{N})\big]|\\
 & \le\frac{|Z_{1}|}{\sqrt{\Delta t_{1}}}\int_{-\sqrt{\Delta t_{1}}}^{\sqrt{\Delta t_{1}}}|\partial_{x_{1}}G(\lambda Z_{1},\sqrt{\Delta t_{2}}Z_{2},\dots,\sqrt{\Delta t_{N}}Z_{N})|d\lambda\\
 & \le|Z_{1}|e^{C\sqrt{t_{N}}|Z|}
\end{align*}
Therefore,
\[
\mathbb{E}(|\xi|^{2})\le\mathbb{E}\big[|Z_{1}|^{2}e^{2C\sqrt{t_{N}}|Z|}\big]=\text{O}(1),\;(\Delta t_{1}\to0).
\]

The proof of (\ref{eq:-37}) is similar to that of (\ref{eq:-36}).

Suppose that $F(S)$ has second order generalised derivatives with
at most polynomial growth a.e.. Using the notation (\ref{eq:-39})
again, the function
\[
K(x_{1},\dots,x_{N})=\tilde{F}(x_{1},\dots,x_{N})J(s,x)^{\text{T}}(Z_{1}Z_{1}^{\text{T}}-I)J(s,x)
\]
has generalised second order derivatives with at most exponential
growth a.e.. Therefore,
\begin{align*}
|\eta_{2}| & =\frac{1}{\Delta t_{1}}|K(\sqrt{\Delta t_{1}}Z_{1},\sqrt{\Delta t_{2}}Z_{2},\dots,\sqrt{\Delta t_{N}}Z_{N})-2K(0,\sqrt{\Delta t_{2}}Z_{2},\dots,\sqrt{\Delta t_{N}}Z_{N})\\
 & \quad+K(-\sqrt{\Delta t_{1}}Z_{1},\sqrt{\Delta t_{2}}Z_{2},\dots,\sqrt{\Delta t_{N}}Z_{N})|\\
 & \le\frac{|Z_{1}|^{2}}{\Delta t_{1}}\int_{0}^{\sqrt{\Delta t_{1}}}\int_{-\lambda_{1}}^{\lambda_{1}}\big|\partial_{x_{1}}^{2}K(\lambda_{2}Z_{1},\sqrt{\Delta t_{2}}Z_{2},\dots,\sqrt{\Delta t_{N}}Z_{N})\big|d\lambda_{2}d\lambda_{1}\\
 & \le|Z_{1}|^{2}e^{C\sqrt{t_{N}}|Z|}.
\end{align*}
Hence,
\[
\mathbb{E}(|\eta_{2}|^{2})\le\mathbb{E}\big[|Z_{1}|^{4}e^{2C\sqrt{t_{N}}|Z|}\big]=\text{O}(1),\;(\Delta t_{1}\to0).
\]
\end{proof}
\begin{rem}
\label{rem:}We would like to point out that the condition of $F$
having second order generalised derivatives is not a foundmental restriction
in practice. Indeed, one may apply payoff smoothing to satisfy this
condition, with the cost of small bias (also known as the overhedging).
\end{rem}

\begin{rem}
\label{rem:-1}Equations (\ref{eq:-23}), (\ref{eq:-25}), (\ref{eq:-26}),
and (\ref{eq:-56}) suggest large variance of the estimagtors in Theorem
\ref{thm:-1} when the covariance matrix is close to being singular,
even if the variance reduction adjustment is applied. We shall apply
the covariance inflation method, as a bias-variance tradeoff, to address
the situation when covariance matrix is close to singularity. See
Section \ref{sec:-2} below.
\end{rem}

It is interested to examine if the existence of second order generalised
derivatives of $F$ in Theorem (\ref{thm:-1}) is necessary. We show
below that for piecewise differentiable functions with finitely many
jumps, the variances of the MC estimators in (\ref{eq:-28}) and (\ref{eq:-29})
does increase to infinity, but with a smaller order, as $\Delta t_{1}\to0$.
\begin{defn}
\label{def:-1}A function $f(x),\;x\in\mathbb{R}^{d}$ is said to
have generalised derivatives in $L^{2}$, if for each $i$ there exists
a $C_{i}>0$ such that 
\[
\Big|\int_{\mathbb{R}^{d}}f(x)\partial_{i}g(x)dx\big]\vert\le C_{i}\Big(\int_{\mathbb{R}^{d}}|\partial_{i}g(x)|^{2}dx\Big)^{1/2}
\]
for any $g\in C^{\infty}(\mathbb{R}^{d})$ with compact support. By
the Riesz representation theorem, this is equivalent to the existence
of $\partial_{i}f\in L^{2}$ such that
\[
\int_{\mathbb{R}^{d}}f(x)\partial_{i}g(x)dx=-\int_{\mathbb{R}^{d}}\partial_{i}f(x)g(x)dx
\]
for any $g\in C^{\infty}(\mathbb{R}^{d})$ with compact support.
\end{defn}

\begin{defn}
\label{def:}A subset of discretisation times $\{t_{i_{1}},\dots,t_{i_{m}}\}\subset\{t_{1},\dots,t_{N}\}$
is called \emph{the set of event times of $F$ }if $F$ can be written
as
\[
F(S)=F(t_{i_{1}},\dots,t_{i_{m}}).
\]
\end{defn}

Lemma \ref{lem:} below gives the variance order as $\Delta t_{1}\to0$
for different regularity assumptions on the payoff. A proof of Lemma
\ref{lem:} is given in Appendix \ref{subsec:-4}. We should point
out that the assumption $t_{1}$ not in event times of a Bermudan
type payoff $F$ can be satisfied by inserting an extra simulation
time point. For American type payoff, one may first approximate the
payoff by a Bermudan version, then insert an extra simulation time
point.
\begin{lem}
\label{lem:}Suppose that the payoff function $F(S)$ does not contain
$t_{1}$ as its event time. 

(i) If $F(S)$ can be written as 
\[
F(S)=F^{s}(S)+\sum_{k=1}^{m}c_{k}\chi_{\{\phi_{k}(S_{t_{2}},\dots,S_{t_{N}})>0\}},
\]
where $F^{s}$ is a function which has second order generalised derivatives
with at most polynomial growth a.e., $c_{k}>0$ are the jump sizes,
and $\phi_{k}$ are smooth functions representing the hyperplane where
jumps occur\footnote{An example of this type of payoffs is barrier options.}.
Then the random variables in (\ref{eq:-43}), (\ref{eq:-44}), and
(\ref{eq:-45}) have variances
\begin{align}
\mathrm{Var}(\xi) & =\mathrm{O}\big((\Delta t_{1})^{-1/2}\big),\;(\Delta t_{1}\to0),\label{eq:-40}\\
\mathrm{Var}(\eta_{1}) & =\mathrm{O}\big((\Delta t_{1})^{-1/2}\big),\;(\Delta t_{1}\to0),\label{eq:-41}\\
\mathrm{Var}(\eta_{2}) & =\mathrm{O}\big((\Delta t_{1})^{-3/2}\big),\;(\Delta t_{1}\to0).\label{eq:-42}
\end{align}

(ii) If $F(S)$ can be written as
\[
F(S)=F^{s}(S)+F^{c}(S),
\]
where $F^{s}$ is a function which has second order generalised derivatives
with at most polynomial growth a.e., and $F^{c}$ is a functional
with first order generalised derivatives with at most polynomial growth
a.e. such that its derivatives can be written as
\[
\partial F^{c}(S)=\sum_{k=1}^{m}c_{k}^{\prime}\chi_{\{\phi_{k}(S_{t_{2}},\dots,S_{t_{N}})>0\}}.
\]
That is, the derivatives $\partial F$ have jumps\footnote{An example of this type of payoffs is vanilla options.}.
Then
\begin{align}
\mathrm{Var}(\xi) & =\mathrm{O}(1),\;(\Delta t_{1}\to0),\label{eq:-60}\\
\mathrm{Var}(\eta_{1}) & =\mathrm{O}(1),\;(\Delta t_{1}\to0),\label{eq:-61}\\
\mathrm{Var}(\eta_{2}) & =\mathrm{O}\big((\Delta t_{1})^{-1/2}\big),\;(\Delta t_{1}\to0).\label{eq:-46}
\end{align}
\end{lem}

\begin{example}
\label{exa:-2}In this example, we verify the order of variance established
in Theorem \ref{thm:-1} and Lemma \ref{lem:}. Consider a one-year
autocallable payoff $F_{0}$ with quarterly up-and-out redemption
and European down-and-in put at maturity. Let $F_{1}$ be the first-order-smoothed
payoff, for which the barrier hit conditions $\chi_{\{R\ge B\}}$and
$\chi_{\{R\le B\}}$ are replaced with
\[
\frac{1}{2}+\frac{1}{2}\sin\Big(\frac{\pi}{2}\cdot\frac{R-B-b}{|b|}\Big),
\]
and
\[
\frac{1}{2}-\frac{1}{2}\sin\Big(\frac{\pi}{2}\cdot\frac{R-B-b}{|b|}\Big),
\]
respectively, where $R$ is the basket performance, $B$ is the corresponding
barrier, and $b$ is a small perturbation to $B$. Moreover, let $F_{2}$
be the second-order-smoothed payoff, for which the barrier hit conditions
are smoothed in the same way as $F_{1}$, and the put option intrinsic
value $(K-R)^{+}$ is replaced with
\[
\left\{ \begin{aligned}\frac{1}{4}\frac{(K+k-R)^{2}}{k}, & \;\text{if}\;|P-K|<k,\\
(K-R)^{+}, & \;\text{otherwise},
\end{aligned}
\right.
\]
where $K$ is the put strike, and $k>0$ is a small perturbation to
$K$. For this example, the simulation time grid is constructed such
that $t_{2},\dots,t_{N}$ are fixed, while $t_{1}$ is variable. Figure
\ref{fig:-2} and Figure \ref{fig:-3} show the $\log$-$\log$ plots
of standard deviation versus $\Delta t_{1}$ for Delta and Gamma,
where ``PW0'', ``PW1'', and ``PW2'' refer to the results with
respect to $F_{0}$, $F_{1}$, and $F_{2}$ respectively. Figure \ref{fig:-4}
and Figure \ref{fig:-5} are zoomed-in versions of Figure \ref{fig:-2}
and Figure \ref{fig:-3} to give a clearer comparison of ``PW1''
and ``PW2''.

However, we should also point out that when the first time step is
not too small (such as $1/360$), payoff smoothing only has limited
variance reduction effect for path weighting method, since this case
is at the left end of Figure \ref{fig:-2} and Figure \ref{fig:-3}.
This is also confirmed in Example \ref{exa:-3} below. 

Numerical test results for this example are generated by the test
case \texttt{test\_variance\_for\_small\_first\_time\_step} in \texttt{pwsen}.

\begin{figure}[H]

\centering{}\includegraphics[scale=0.7]{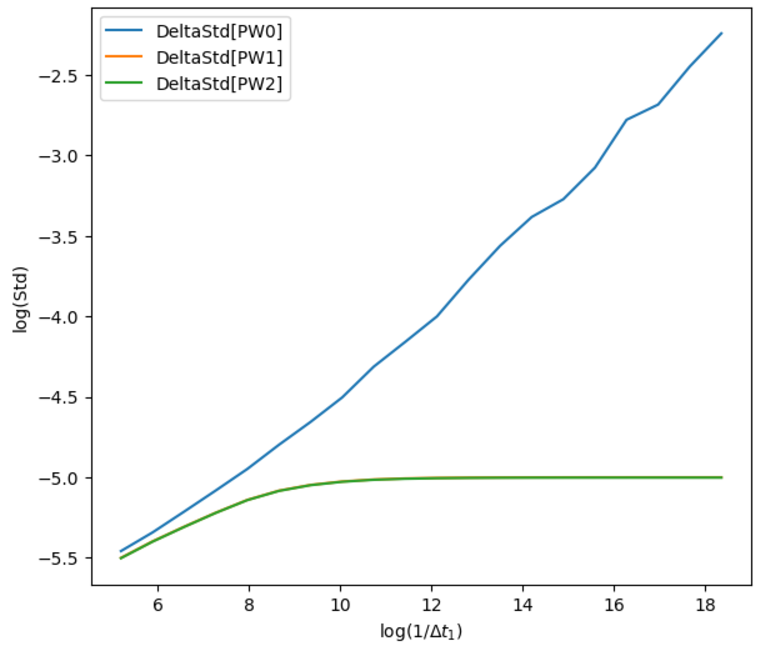}\caption{Standard Deviation of Delta versus $\log(1/\Delta t_{1})$}
\label{fig:-2}
\end{figure}

\begin{figure}[H]
\centering{}\includegraphics[scale=0.7]{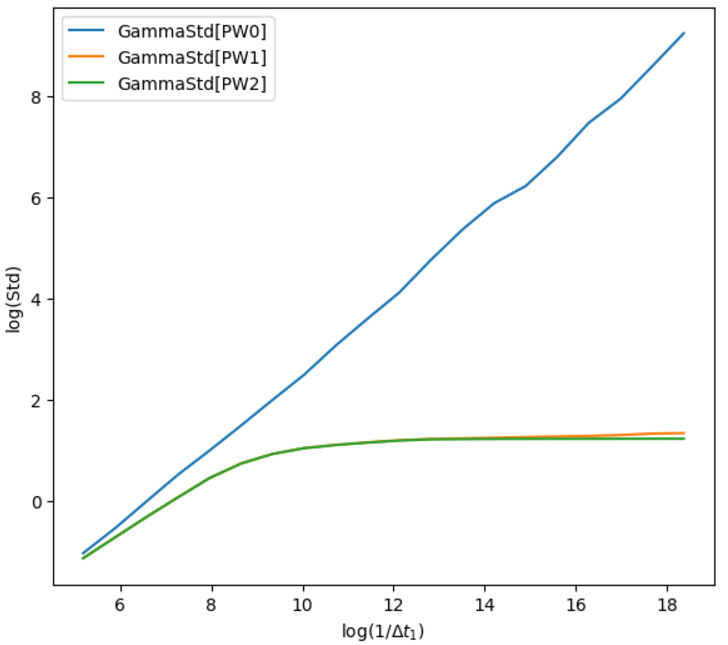}\caption{Standard Deviation of Gamma versus $\log(1/\Delta t_{1})$}
\label{fig:-3}
\end{figure}

\begin{figure}[H]
\centering{}\includegraphics[scale=0.7]{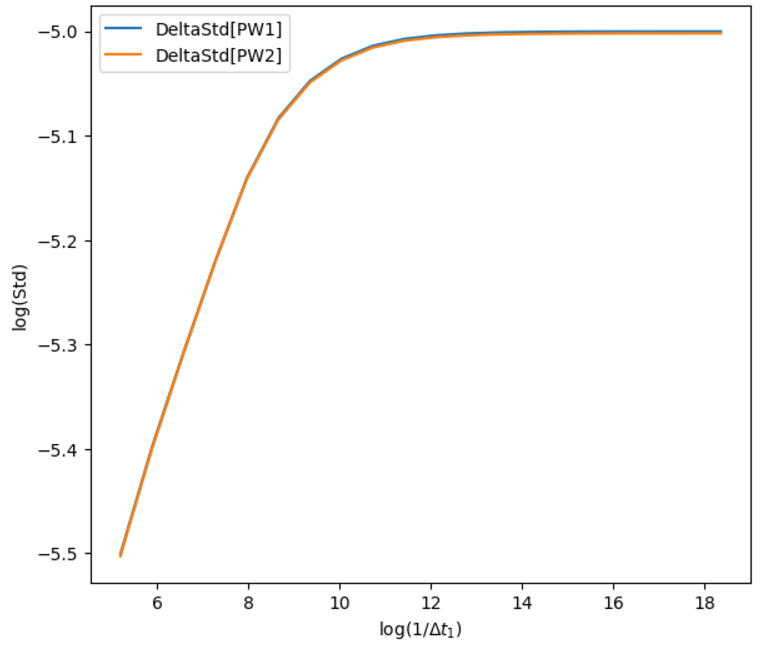}\caption{Standard Deviation of Delta versus $\log(1/\Delta t_{1})$}
\label{fig:-4}
\end{figure}
\begin{figure}[H]
\centering{}\includegraphics[scale=0.7]{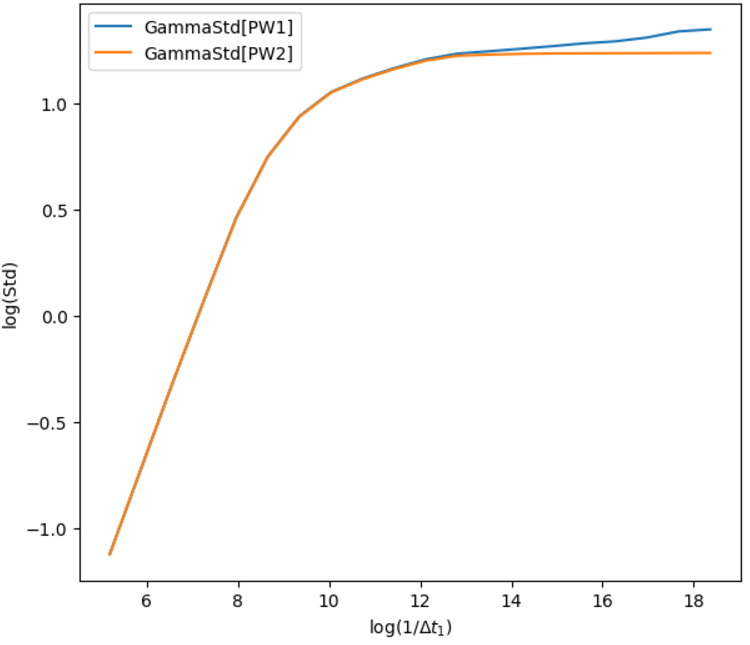}\caption{Standard Deviation of Gamma versus $\log(1/\Delta t_{1})$}
\label{fig:-5}
\end{figure}
\end{example}

\begin{example}
\label{exa:-3}Table \ref{tab:-1} shows the standard deviations of
Delta and Gamma computed by finite difference (FD), raw path weighting
(RPW), and adjusted path weighting (PW) under various volatility and
correlation levels. The payoffs being tested are a one-year autocallable
with basket size $2$ and a smoothed version of the same autocallable.
The MC simulation uses 500k paths for each method, and uses first
time step $\Delta t_{1}=1/360$ and $\Delta t_{i}=7/360$ for the
rest. The columns with ``(PS)'' in Table \ref{tab:-1} refer the
results for test cases where payoff smoothing is applied. For finite
difference, the spot bump being used is the expected one day spot
percentage change $\sigma/\sqrt{360}$. 

As shown by Table \ref{tab:-1}, for Delta of unsmoothed payoff, PW
has slightly higher standard deviation than FD for most of the cases,
while the difference becomes more significant when correlation is
close to $1$ (similar result holds for correlation $-1$). This is
due to the appearance of the inverse correlation matrix in the path
weighting formulae. For all cases, there is huge variance reduction
from RPW to PW. The variance reduction effect of payoff smoothing
is less significant for PW than for FD, especially when volatility
is small.

As shown by Table \ref{tab:-2}, for Gamma of unsmoothed payoff, PW
has smaller standard deviation than FD for most of the cases, while
PW has larger standard deviation than FD when correlation is close
to $1$ (similar result holds for correlation $-1$). For all cases,
there is huge variance reduction from RPW to PW. Payoff smoothing
is much less effective for PW than for FD, just as pointed out in
Example \ref{exa:-2}.

Numerical test results in this example are generated by the test case
\texttt{test\_variance\_reduction} in \texttt{pwsen}.

\begin{table}[H]
\begin{centering}
\begin{tabular}{|c|c|c|c|c|c|c|c|}
\hline 
\textbf{Volatility} & \textbf{Correlation} & \textbf{FD} & \textbf{RPW} & \textbf{PW} & \textbf{FD (PS)} & \textbf{RPW (PS)} & \textbf{PW (PS)}\tabularnewline
\hline 
\hline 
0.05 & 0.0 & 0.0003  & 0.0537 & 0.0003 & 9e-05 & 0.0536 & 0.0001\tabularnewline
\hline 
0.05 & 0.3 & 0.0003 & 0.0563 & 0.0003 & 9e-05 & 0.0562 & 0.0001\tabularnewline
\hline 
0.05 & 0.5 & 0.0003 & 0.0620 & 0.0003 & 9e-05 & 0.0619 & 0.0001\tabularnewline
\hline 
0.05 & 0.7 & 0.0003 & 0.0752 & 0.0004 & 8e-05 & 0.0751 & 0.0002\tabularnewline
\hline 
0.05 & 0.95 & 0.0002 & 0.1720 & 0.0007 & 8e-05 & 0.1719 & 0.0003\tabularnewline
\hline 
0.1 & 0.0 & 0.0013 & 0.0298 & 0.0016 & 0.0008 & 0.0292 & 0.0013\tabularnewline
\hline 
0.1 & 0.3 & 0.0013 & 0.0310 & 0.0016 & 0.0008 & 0.0305 & 0.0013\tabularnewline
\hline 
0.1 & 0.5 & 0.0012 & 0.03402 & 0.0016 & 0.0007 & 0.0335 & 0.0013\tabularnewline
\hline 
0.1 & 0.7 & 0.0012 & 0.0409 & 0.0018 & 0.0007 & 0.0403 & 0.0015\tabularnewline
\hline 
0.1 & 0.95 & 0.0011 & 0.0918 & 0.0033 & 0.0006 & 0.0908 & 0.0026\tabularnewline
\hline 
0.2 & 0.0 & 0.0019 & 0.0229 & 0.0022 & 0.0014 & 0.0224 & 0.0021\tabularnewline
\hline 
0.2 & 0.3 & 0.0018 & 0.0234 & 0.0022 & 0.0014 & 0.0229 & 0.0021 \tabularnewline
\hline 
0.2 & 0.5 & 0.0018 & 0.0252  & 0.0023 & 0.0013 & 0.0247 & 0.0021\tabularnewline
\hline 
0.2 & 0.7 & 0.0017  & 0.0296 & 0.0025 & 0.0013 & 0.0290 & 0.0024\tabularnewline
\hline 
0.2 & 0.95 & 0.0016 & 0.0630 & 0.0050 & 0.0012 & 0.0619 & 0.0045\tabularnewline
\hline 
0.3 & 0.0 & 0.0020  & 0.0212 & 0.0024 & 0.0017 & 0.0209 & 0.0023\tabularnewline
\hline 
0.3 & 0.3 & 0.0020  & 0.0215 & 0.0024 & 0.0016 & 0.0212 & 0.0022\tabularnewline
\hline 
0.3 & 0.5 & 0.0019 & 0.0229 & 0.0025 & 0.0016 & 0.0226 & 0.0023\tabularnewline
\hline 
0.3 & 0.7 & 0.0018 & 0.0267 & 0.0028 & 0.0015 & 0.0263 & 0.0026\tabularnewline
\hline 
0.3 & 0.95 & 0.0017 & 0.0556 & 0.0054 & 0.0014 & 0.0548 & 0.0051\tabularnewline
\hline 
0.5 & 0.0 & 0.0021 & 0.0193 & 0.0023 & 0.0018 & 0.0192 & 0.0023\tabularnewline
\hline 
0.5 & 0.3 & 0.0020 & 0.0194 & 0.0023 & 0.0018 & 0.0192 & 0.0023\tabularnewline
\hline 
0.5 & 0.5 & 0.0019 & 0.0206 & 0.0024 & 0.0018 & 0.0204 & 0.0024\tabularnewline
\hline 
0.5 & 0.7 & 0.0019  & 0.0239 & 0.0028 & 0.0017 & 0.0237 & 0.0027\tabularnewline
\hline 
0.5 & 0.95 & 0.0018 & 0.0492 & 0.0055 & 0.0016 & 0.0487 & 0.0053\tabularnewline
\hline 
\end{tabular}
\par\end{centering}
\caption{Delta Standard Deviation Comparison }
\label{tab:-1}

\end{table}
\begin{table}[H]
\begin{centering}
\begin{tabular}{|c|c|c|c|c|c|c|c|}
\hline 
\textbf{Volatility} & \textbf{Correlation} & \textbf{FD} & \textbf{RPW} & \textbf{PW} & \textbf{FD (PS)} & \textbf{RPW (PS)} & \textbf{PW (PS)}\tabularnewline
\hline 
\hline 
0.05 & 0.0 & 0.4511 & 28.82 & 0.1873 & 0.0139 & 28.81 & 0.0403\tabularnewline
\hline 
0.05 & 0.3 & 0.4199  & 31.66 & 0.1720 & 0.0141 & 31.65 & 0.0295\tabularnewline
\hline 
0.05 & 0.5 & 0.3851 & 38.39 & 0.2119 & 0.0133 & 38.38 & 0.0462\tabularnewline
\hline 
0.05 & 0.7 & 0.4191 & 56.44 & 0.2879 & 0.0127 & 56.42 & 0.0700\tabularnewline
\hline 
0.05 & 0.95 & 0.3473 & 295.3 & 1.129 & 0.0149 & 295.2 & 0.1521\tabularnewline
\hline 
0.1 & 0.0 & 0.9658 & 7.982 & 0.4118 & 0.1075 & 7.835 & 0.2442\tabularnewline
\hline 
0.1 & 0.3 & 0.9408 & 8.733 & 0.4280 & 0.1038 & 8.594 & 0.2588\tabularnewline
\hline 
0.1 & 0.5 & 0.9074 & 10.53  & 0.4708 & 0.1023 & 10.37 & 0.2783\tabularnewline
\hline 
0.1 & 0.7 & 0.8716 & 15.37 & 0.6161 & 0.0997 & 15.15 & 0.3295\tabularnewline
\hline 
0.1 & 0.95 & 0.7738 & 78.87 & 2.537 & 0.0986 & 77.95 & 1.067\tabularnewline
\hline 
0.2 & 0.0 & 0.6888 & 3.078 & 0.2925 & 0.1959 & 3.016 & 0.2411 \tabularnewline
\hline 
0.2 & 0.3 & 0.6764 & 3.299 & 0.3063 & 0.1880 & 3.235 & 0.2459\tabularnewline
\hline 
0.2 & 0.5 & 0.6575 & 3.908 & 0.3452 & 0.1819 & 3.829 & 0.2775\tabularnewline
\hline 
0.2 & 0.7 & 0.6261 & 5.547 & 0.4685 & 0.1744 & 5.444 & 0.3522\tabularnewline
\hline 
0.2 & 0.95 & 0.5804  & 26.95 & 1.927 & 0.1619 & 26.51 & 1.268\tabularnewline
\hline 
0.3 & 0.0 & 0.4979 & 1.902 & 0.2097 & 0.2215 & 1.876 & 0.1825\tabularnewline
\hline 
0.3 & 0.3 & 0.4831  & 2.022 & 0.2167 & 0.2132 & 1.994 & 0.1867\tabularnewline
\hline 
0.3 & 0.5 & 0.4789 & 2.375 & 0.2460 & 0.2089 & 2.342 & 0.2091\tabularnewline
\hline 
0.3 & 0.7 & 0.4542 & 3.345 & 0.3290 & 0.2014 & 3.298 & 0.2780\tabularnewline
\hline 
0.3 & 0.95 & 0.4265 & 15.86  & 1.433 & 0.1880 & 15.63 & 1.085\tabularnewline
\hline 
0.5 & 0.0 & 0.3121 & 1.040 & 0.1265 & 0.2045 & 1.031 & 0.1173\tabularnewline
\hline 
0.5 & 0.3 & 0.2965 & 1.096 & 0.1312 & 0.1971 & 1.087 & 0.1190\tabularnewline
\hline 
0.5 & 0.5 & 0.2906 & 1.282 & 0.1492 & 0.1948 & 1.271 & 0.1356\tabularnewline
\hline 
0.5 & 0.7 & 0.2864 & 1.797 & 0.2060 & 0.1864 & 1.781 & 0.1809\tabularnewline
\hline 
0.5 & 0.95 & 0.2658  & 8.436 & 0.8833 & 0.1757 & 8.345 & 0.7533\tabularnewline
\hline 
\end{tabular}
\par\end{centering}
\caption{Gamma Standard Deviation Comparison}
\label{tab:-2}
\end{table}
\end{example}

\begin{example}
\label{exa:-4}Figure \ref{fig:-9} shows the calculation time ratio
of full Delta vector and full Gamma matrix using finite difference
approximation (FD) and path weighting (PW), as a function of basket
size. The MC simulation uses 100k paths for each test case. In Figure
\ref{fig:-9}, ``DeltaTimeRatio'' refers to the calculation time
of Delta using FD divided by that using PW, and ``GammaTimeRatio''
referes to the corresponding time ratio for Gamma.

Numerical results in this example are generated by the test case \texttt{test\_calculation\_time}
in \texttt{pwsen}.

\begin{figure}[H]

\begin{centering}
\includegraphics[scale=0.7]{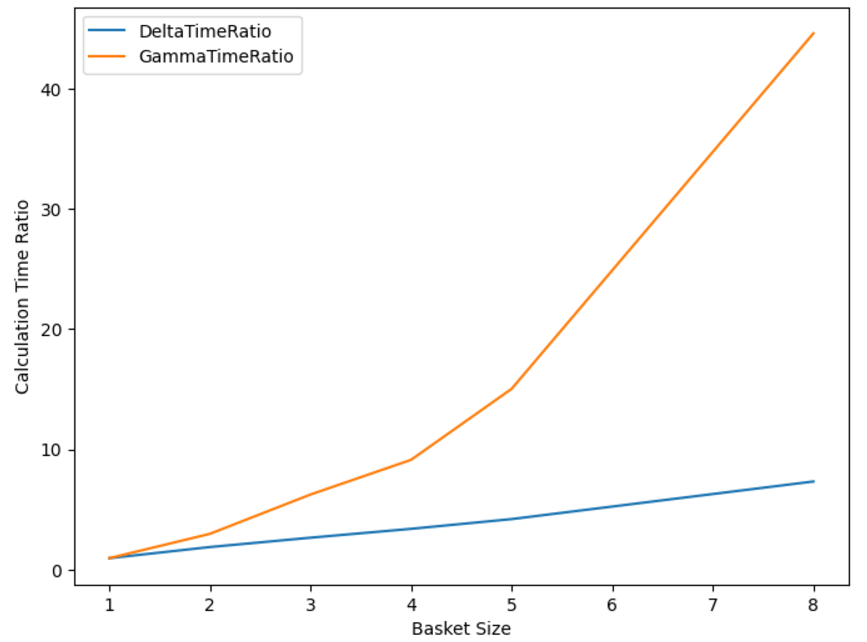}\caption{Calculation Time Ratio}
\label{fig:-9}
\par\end{centering}
\end{figure}
\end{example}

\section{\label{sec:-2}Covaraince Inflation: A Bias--Variance Tradeoff}

In this section, we address the degenerate case when the covariance
matrix is singular (either volatility is zero or correlation matrix
does not have full rank), for which, as pointed out in Remark \ref{rem:-1},
the formulae (\ref{eq:-28}) and (\ref{eq:-29}) become singular.
A natural approach would be to reduce to a subset of stocks such that
the covariance matrix is non-singular. However, such approach lacks
of generality in terms of implementation. It is also prone to large
variance when the covariance is not singular but close to being singular.
The approach we adopt here is to apply inflation to the covariance
so that it is away from singularity. Thereby, the variance of the
estimator is reduced significantly, with the cost of small bias introduced. 

The following theorem gives the order of bias introduced by perturbing
the driver noise $Z_{1}$ at $t_{1}$ only. A proof of Theorem \ref{thm:-2}
is given in Appendix \ref{subsec:-5}.
\begin{thm}
\label{thm:-2}Let $S_{t}$ and $\tilde{S}_{t}$ be the processes
defined by 
\begin{equation}
\begin{aligned}S_{t_{i+1}} & =H_{i+1}(S_{t_{i}},\sqrt{\Delta t_{i+1}}Z_{i+1}),\;i\ge1,\\
S_{t_{1}} & =H_{1}(s,\sqrt{\Delta t_{1}}Z_{1}),\\
S_{t_{0}} & =s,
\end{aligned}
\label{eq:-54}
\end{equation}
and
\begin{equation}
\begin{aligned}\tilde{S}_{t_{i+1}} & =H_{i+1}(\tilde{S}_{t_{i}},\sqrt{\Delta t_{i+1}}Z_{i+1}),\;i\ge1,\\
\tilde{S}_{t_{1}} & =H_{1}(s,\sqrt{\Delta t_{1}}(Z_{1}+\epsilon Y)),\\
\tilde{S}_{t_{0}} & =s,
\end{aligned}
\label{eq:-53}
\end{equation}
where $Y$ is an $\mathbb{R}^{n}$-valued random variable such that
$Y$ is independent of $Z_{2},\dots,Z_{N}$, and
\[
\mathbb{E}(Y)=0,
\]
and
\[
\mathbb{E}(e^{\lambda|Y|})<\infty
\]
for any $\lambda>0$. Let $F(S)=F(S_{t_{2}},\dots,S_{t_{N}})$ be
a payoff function which does not contain $t_{1}$ as an event date.
Then
\[
\big|\mathbb{E}(F(S_{t_{2}},\dots,S_{t_{N}}))-\mathbb{E}(F(\tilde{S}_{t_{2}},\dots,\tilde{S}_{t_{N}}))\big|\le C\epsilon\Delta t_{1},
\]
where $C>0$ is a constant depending on $F,$ $H$, and $Y$.
\end{thm}

\begin{rem}
Recall that the MC bias introduced by SDE discretisation is $\text{O}(\max_{i}\Delta t_{i})$.
Theorem \ref{thm:-2} shows that the bias introduced by perturbing
the driver noise at $t_{1}$is just of the same order.
\end{rem}

\begin{defn}
\label{def:-2}For any $\epsilon_{1}>0$ and $\epsilon_{2}>0$, the
\emph{covariance inflated process} $S_{t}^{\epsilon_{1},\epsilon_{2}}$
is defined to be the process given by 
\begin{equation}
\begin{aligned}S_{t_{i+1}}^{\epsilon_{1},\epsilon_{2}} & =H_{i}(S_{t_{i}}^{\epsilon_{1},\epsilon_{2}},\sqrt{\Delta t_{i+1}}Z_{i+1}),\;i\ge1,\\
S_{t_{1}}^{\epsilon_{1},\epsilon_{2}} & =H_{1}^{\epsilon_{1},\epsilon_{2}}(s,\sqrt{\Delta t_{1}}Z_{1}),
\end{aligned}
\label{eq:-52}
\end{equation}
where the functional $H_{1}^{\epsilon_{1},\epsilon_{2}}$ is obtained
by replacing the volatility curve (for the time step $[t_{0},t_{1})$
only) $\sigma_{0}^{(k)}(s^{(k)})=\sigma^{(k)}(t_{0},s^{(k)})$ by
\begin{equation}
\sigma_{0,\epsilon_{1}}^{(k)}(t,s^{(k)})=\sigma_{0}^{(k)}(S_{t_{0}}^{(k)})+\epsilon_{1},\label{eq:-50}
\end{equation}
and replacing the correlation matrix $\Sigma=Q\Lambda^{2}Q^{\text{T}}$
(for the time step $[t_{0},t_{1})$ only) by 
\begin{equation}
\Sigma_{\epsilon_{2}}=\epsilon_{2}I+(1-\epsilon_{2})\Sigma.\label{eq:-51}
\end{equation}
\end{defn}

\begin{rem}
By Theorem \ref{thm:-2}, the bias introduced by covariance inflation,
i.e. replacing the process $S_{t}$ with $S_{t}^{\epsilon_{1},\epsilon_{2}}$,
is bounded by $C(\epsilon_{1}+\epsilon_{2})\Delta t_{1}$.
\end{rem}

\begin{rem}
It is possible to combine Theorem \ref{thm:-1} and Theorem \ref{thm:-2}
to derive the optimal $\epsilon_{1}$ and $\epsilon_{2}$ which minimise
the mean square error of (\ref{eq:-28}) and (\ref{eq:-29}). Such
formula for the optimal $\epsilon_{1},\epsilon_{2}$ is not very useful
in practice. Because the constants in the error bounds depend on the
payoff function and are usually unknown. 

Our approach to address close-to-singularity covariance matrices (either
correlation is close to being singular or volatility is close to zero)
is to insert an extra small simulation time point ($1/360$ for example),
and then apply Theorem \ref{thm:} to the covariance inflated process
$S_{t}^{\epsilon_{1},\epsilon_{2}}$, by which we sacrifice the property
of being unbiased for the ability to handle singular covariance matrix
and also a smaller MC variance.

For tests in Example \ref{exa:-5} below, we use the following covariance
inflation
\begin{align*}
\epsilon_{1} & =0.01,\\
\epsilon_{2} & =0.5\cdot e^{-10\bar{\lambda}},
\end{align*}
where
\begin{align*}
\bar{\lambda} & =\min_{1\le k\le n}|\Lambda_{kk}|,
\end{align*}
with $\Lambda_{kk}$ being the $(k,k)$-element of the diagoanl matrix
$\Lambda$ in the eigenvalue decomposition $\Sigma=Q\Lambda^{2}Q^{\text{T}}$
of the correlation matrix. Table \ref{tab:} shows the value of coefficient
$\epsilon_{2}$ for various correlation $\rho$ for 2-dimension case

\begin{table}[H]
\begin{centering}
\begin{tabular}{|c|c|}
\hline 
\textbf{Correlation ${\bf \rho}$} & \textbf{Coefficient ${\bf \epsilon_{2}}$}\tabularnewline
\hline 
\hline 
0.0 & 0.0\tabularnewline
\hline 
0.3 & 1e-4\tabularnewline
\hline 
0.5 & 4e-4\tabularnewline
\hline 
0.7 & 0.002\tabularnewline
\hline 
0.8 & 0.005\tabularnewline
\hline 
0.9 & 0.02\tabularnewline
\hline 
0.95 & 0.05\tabularnewline
\hline 
1.0 & 0.5\tabularnewline
\hline 
\end{tabular}\caption{Correlation Inflation Coefficient}
\label{tab:}
\par\end{centering}
\end{table}
\end{rem}

\begin{example}
\label{exa:-5}As a test on the effect of covariance inflation, Table
\ref{tab:-3} and Table \ref{tab:-4} show the results for the same
test as in Example \ref{exa:-3}, with an additional method path weighting
with covariance inflation (PWCI) added. Moreover, Figure \ref{fig:-10}
and Figure \ref{fig:-11} show the convergence paths for the two extreme
cases where correlation is $0.95$ and volatility is $0.05$ and $0.5$
respectively.

Numerical test results in this example are generated by the test case
\texttt{test\_variance\_reduction} in \texttt{pwsen}.

\begin{table}[H]
\begin{centering}
\begin{tabular}{|c|c|c|c|c|}
\hline 
\textbf{Volatility} & \textbf{Correlation} & \textbf{FD} & \textbf{PW} & \textbf{PWCI}\tabularnewline
\hline 
\hline 
0.05 & 0.0 & 0.0003 & 0.0003 & 0.0002\tabularnewline
\hline 
0.05 & 0.3 & 0.0003 & 0.0003 & 0.0003\tabularnewline
\hline 
0.05 & 0.5 & 0.0003 & 0.0003 & 0.0003\tabularnewline
\hline 
0.05 & 0.7 & 0.0003 & 0.0004 & 0.0003\tabularnewline
\hline 
0.05 & 0.95 & 0.0002 & 0.0007 & 0.0003\tabularnewline
\hline 
0.1 & 0.0 & 0.0013 & 0.0016 & 0.0015\tabularnewline
\hline 
0.1 & 0.3 & 0.0013 & 0.0016 & 0.0015\tabularnewline
\hline 
0.1 & 0.5 & 0.0012 & 0.0016 & 0.0015 \tabularnewline
\hline 
0.1 & 0.7 & 0.0012 & 0.0018 & 0.0016\tabularnewline
\hline 
0.1 & 0.95 & 0.0011 & 0.0033 & 0.0014\tabularnewline
\hline 
0.2 & 0.0 & 0.0019 & 0.0022 & 0.0021\tabularnewline
\hline 
0.2 & 0.3 & 0.0018 & 0.0022 & 0.0021\tabularnewline
\hline 
0.2 & 0.5 & 0.0018 & 0.0023 & 0.0022\tabularnewline
\hline 
0.2 & 0.7 & 0.0017 & 0.0025 & 0.0024\tabularnewline
\hline 
0.2 & 0.95 & 0.0016 & 0.0050 & 0.0021\tabularnewline
\hline 
0.3 & 0.0 & 0.0020 & 0.0024 & 0.0023\tabularnewline
\hline 
0.3 & 0.3 & 0.0020 & 0.0024 & 0.0023\tabularnewline
\hline 
0.3 & 0.5 & 0.0019 & 0.0025 & 0.0024\tabularnewline
\hline 
0.3 & 0.7 & 0.0018 & 0.0028 & 0.0026\tabularnewline
\hline 
0.3 & 0.95 & 0.0017 & 0.0054 & 0.0024\tabularnewline
\hline 
0.5 & 0.0 & 0.0021 & 0.0023 & 0.0023\tabularnewline
\hline 
0.5 & 0.3 & 0.0020 & 0.0023 & 0.0023\tabularnewline
\hline 
0.5 & 0.5 & 0.0019 & 0.0024 & 0.0024\tabularnewline
\hline 
0.5 & 0.7 & 0.0019 & 0.0028 & 0.0027\tabularnewline
\hline 
0.5 & 0.95 & 0.0018 & 0.0055 & 0.0025\tabularnewline
\hline 
\end{tabular}
\par\end{centering}
\caption{Delta Standard Deviation Comparison}
\label{tab:-3}
\end{table}
\begin{table}[H]
\begin{centering}
\begin{tabular}{|c|c|c|c|c|}
\hline 
\textbf{Volatility} & \textbf{Correlation} & \textbf{FD} & \textbf{PW} & \textbf{PWCI}\tabularnewline
\hline 
\hline 
0.05 & 0.0 & 0.4511 & 0.1873 & 0.1367\tabularnewline
\hline 
0.05 & 0.3 & 0.4199 & 0.1720 & 0.1407\tabularnewline
\hline 
0.05 & 0.5 & 0.3851 & 0.2119 & 0.1518 \tabularnewline
\hline 
0.05 & 0.7 & 0.4191 & 0.2879 & 0.2079\tabularnewline
\hline 
0.05 & 0.95 & 0.3473 & 1.129 & 0.1690\tabularnewline
\hline 
0.1 & 0.0 & 0.9658 & 0.4118 & 0.3458\tabularnewline
\hline 
0.1 & 0.3 & 0.9408 & 0.4280 & 0.3647\tabularnewline
\hline 
0.1 & 0.5 & 0.9074 & 0.4708 & 0.4016\tabularnewline
\hline 
0.1 & 0.7 & 0.8716 & 0.6161 & 0.5044\tabularnewline
\hline 
0.1 & 0.95 & 0.7738 & 2.537 & 0.4387\tabularnewline
\hline 
0.2 & 0.0 & 0.6888 & 0.2925 & 0.2694\tabularnewline
\hline 
0.2 & 0.3 & 0.6764 & 0.3063 & 0.2791\tabularnewline
\hline 
0.2 & 0.5 & 0.6575 & 0.3452 & 0.3132\tabularnewline
\hline 
0.2 & 0.7 & 0.6261 & 0.4685 & 0.4149\tabularnewline
\hline 
0.2 & 0.95 & 0.5804 & 1.927 & 0.3566\tabularnewline
\hline 
0.3 & 0.0 & 0.4979 & 0.2097 & 0.1974\tabularnewline
\hline 
0.3 & 0.3 & 0.4831 & 0.2167 & 0.2043\tabularnewline
\hline 
0.3 & 0.5 & 0.4789 & 0.2460 & 0.2308\tabularnewline
\hline 
0.3 & 0.7 & 0.4542 & 0.3290 & 0.2986\tabularnewline
\hline 
0.3 & 0.95 & 0.4265 & 1.433 & 0.2714\tabularnewline
\hline 
0.5 & 0.0 & 0.3121 & 0.1265 & 0.1220\tabularnewline
\hline 
0.5 & 0.3 & 0.2965 & 0.1312 & 0.1263\tabularnewline
\hline 
0.5 & 0.5 & 0.2906 & 0.1492 & 0.1443\tabularnewline
\hline 
0.5 & 0.7 & 0.2864 & 0.2060 & 0.1911\tabularnewline
\hline 
0.5 & 0.95 & 0.2658 & 0.8833 & 0.1751\tabularnewline
\hline 
\end{tabular}
\par\end{centering}
\caption{Gamma Standard Deviation Comparison}
\label{tab:-4}
\end{table}

\begin{figure}[H]
\centering{}%
\begin{tabular}{cc}
\includegraphics[scale=0.5]{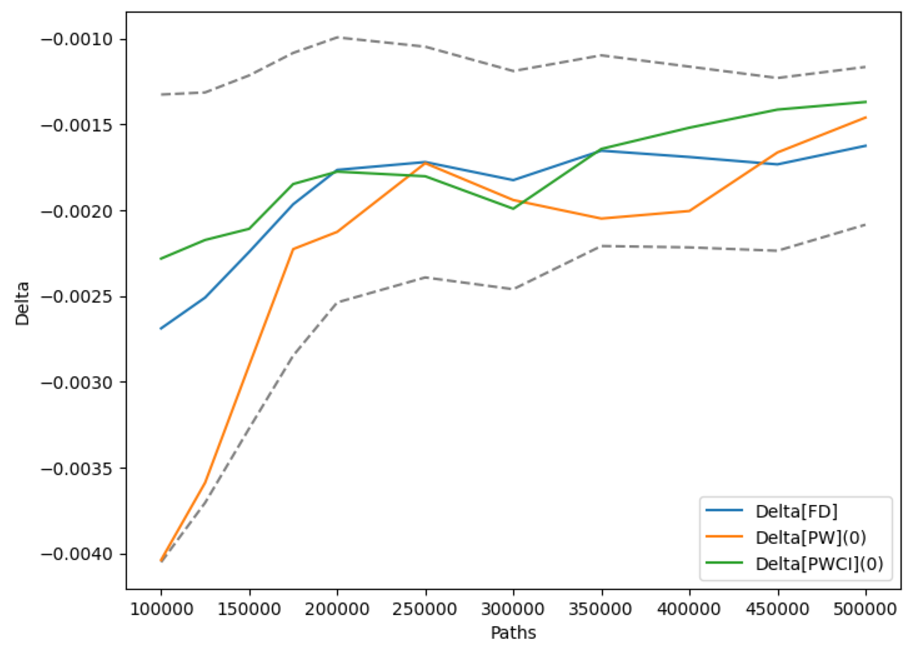} & \includegraphics[scale=0.5]{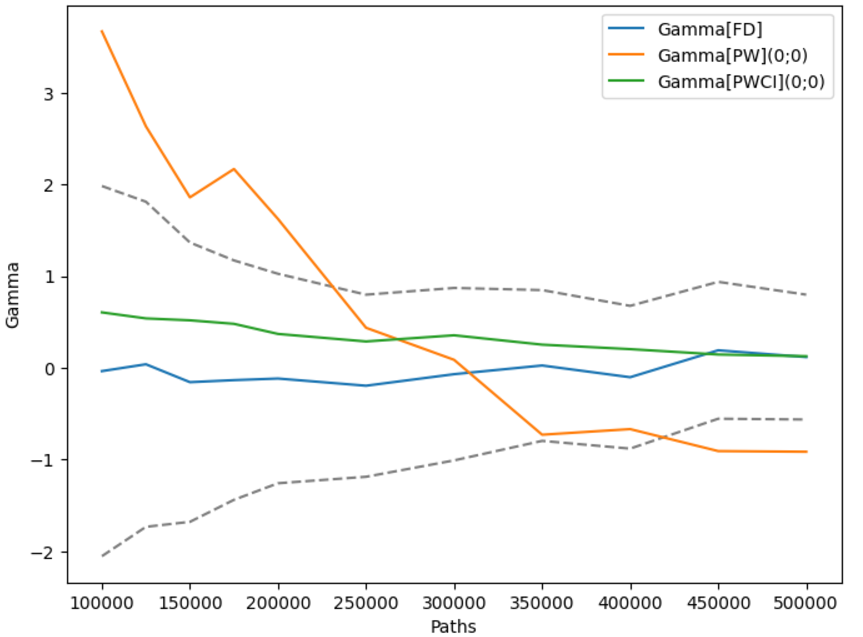}\tabularnewline
\end{tabular}\caption{Convergence Path, $\sigma=0.05,\rho=0.95$}
\label{fig:-10}
\end{figure}

\begin{figure}[H]
\centering{}%
\begin{tabular}{cc}
\includegraphics[scale=0.5]{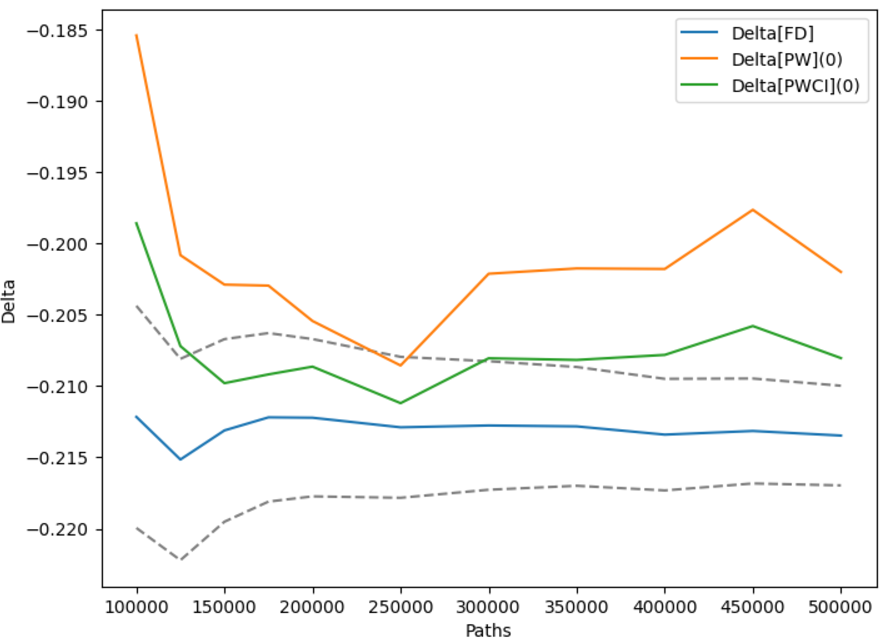} & \includegraphics[scale=0.5]{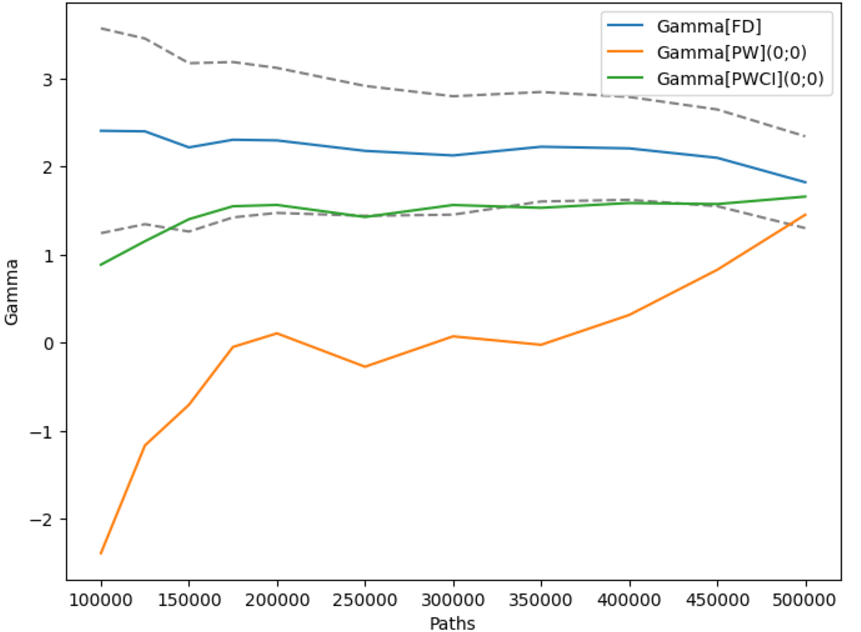}\tabularnewline
\end{tabular}\caption{Convergence Path, $\sigma=0.5,\rho=0.95$}
\label{fig:-11}
\end{figure}
\end{example}

\section{Appendix\label{sec:-3}}

\subsection{\label{subsec:-2}Proof of Theorem \ref{thm:-3}}

We provide the derivation of Theorem \ref{thm:-3} below. The following
integration by parts formula is crucial to derivation of the path
weighting formula.
\begin{lem}
\label{lem:-3}Let $X$ be a real-valued random variable having normal
distribution with variance $\sigma^{2}$. For any $F,G\in C^{1}(\mathbb{R})$
with at most exponential growth as $|x|\to\infty$, it holds that
\[
\mathbb{E}[F^{\prime}(X)G(X)]=\mathbb{E}\Big[F(X)\Big(\frac{X}{\sigma^{2}}G(X)-G^{\prime}(X)\Big)\Big].
\]
\end{lem}

We can now prove Theorem \ref{thm:-3}.
\begin{proof}[Proof of Theorem \ref{thm:-3}]
We may assume $F\in C^{2}$. The proof of the general case follows
from a standard approximation by convolutions argument. Denote $X_{i}^{(k)}=\sqrt{\Delta t_{i}}Z_{i}^{(k)}$.
Then
\begin{equation}
\begin{aligned}\partial_{X_{i}^{(l)}}S_{t_{j+1}}^{(k)} & =\sum_{p=1}^{n}\partial_{X_{i}^{(l)}}S_{t_{j}}^{(p)}\cdot\partial_{s^{(p)}}H_{j}^{(k)}(S_{t_{j}},X_{j+1}),\;j\ge i,\\
\partial_{X_{i}^{(l)}}S_{t_{i}}^{(k)} & =\partial_{x^{(l)}}H_{i-1}^{(k)}(S_{t_{i-1}},X_{i}).
\end{aligned}
\label{eq:-1-1-1}
\end{equation}
\begin{equation}
\begin{aligned}\partial_{s^{(l)}}S_{t_{j+1}}^{(k)} & =\sum_{p=1}^{n}\partial_{s^{(l)}}S_{t_{j}}^{(p)}\cdot\partial_{s^{(p)}}H_{j}^{(k)}(S_{t_{j}},X_{j+1}),\;j\ge1.\\
\partial_{s^{(l)}}S_{t_{1}}^{(k)} & =\partial_{s^{(l)}}H_{0}^{(k)}(s,X_{1}).
\end{aligned}
\label{eq:-2-1-1}
\end{equation}
Let $\partial_{s}H_{0}$ be the matrix given by
\[
(\partial_{s}H_{0})_{kp}=\partial_{s^{(p)}}H_{0}^{(k)},
\]
and similarly
\[
(\partial_{x}H_{0})_{kp}=\partial_{x^{(p)}}H_{0}^{(k)}.
\]
Then
\[
\partial_{s}S_{t_{j}}=(\partial_{X_{1}}S_{t_{j}})(\partial_{x}H_{0})^{-1}(\partial_{s}H_{0})=\partial_{X_{1}}S_{t_{j}}\cdot J(s,X_{1}).
\]
Therefore,
\[
\begin{aligned}\partial_{s^{(l)}}f(s) & =\mathbb{E}\Big[\sum_{j=1}^{N}\sum_{k=1}^{n}\partial_{s_{j}^{(k)}}F(S)\partial_{s^{(l)}}S_{t_{j}}^{(k)}\Big]\\
 & =\mathbb{E}\Big[\sum_{j=1}^{N}\sum_{1\le k,p\le n}\partial_{s_{j}^{(k)}}F(S)\partial_{X_{1}^{(p)}}S_{t_{j}}^{(k)}J_{pl}(s,X_{1})\Big]\\
 & =\mathbb{E}\Big[\sum_{p=1}^{n}\partial_{x_{1}^{(p)}}G(X)J_{pl}(s,X_{1})\Big].
\end{aligned}
\]
By integration by parts,
\[
\mathbb{E}[\partial_{x_{1}^{(p)}}G(X)J_{pl}(s,X_{1})]=\mathbb{E}\Big[G(X)\Big(\frac{X_{1}^{(p)}}{\Delta t_{1}}J_{pl}(s,X_{1})-\partial_{x^{(p)}}J_{pl}(s,X_{1})\Big)\Big].
\]
Therefore,
\[
\begin{aligned}\partial_{s^{(l)}}f(s) & =\mathbb{E}\Big[F(S)\sum_{p=1}^{n}\Big(\frac{Z_{1}^{(p)}}{\sqrt{\Delta t_{1}}}J_{pl}(s,\sqrt{\Delta}Z_{1})-\partial_{x^{(p)}}J_{pl}(s,\sqrt{\Delta}Z_{1})\Big)\Big]\end{aligned}
.
\]
This proves (\ref{eq:-19-1}).

Similarly, we have
\[
\begin{aligned}\partial_{s^{(l)}s^{(m)}}^{2}f(s) & =\mathbb{E}\Big[\sum_{j}^{N}\sum_{k,p}\partial_{s_{j}^{(k)}}F(S)\partial_{s^{(m)}}S_{t_{j}}^{(k)}\Big(\frac{X_{1}^{(p)}}{\Delta t_{1}}J_{pl}(s,X_{1})-\partial_{x^{(p)}}J_{pl}(s,X_{1})\Big)\Big]\\
 & \quad+\mathbb{E}\Big[F(S)\sum_{p}\Big(\frac{X_{1}^{(p)}}{\Delta t_{1}}\partial_{s^{(m)}}J_{pl}(s,X_{1})-\partial_{s^{(m)}x^{(p)}}^{2}J_{pl}(s,X_{1})\Big)\Big]\\
 & =\mathbb{E}\Big[\sum_{j=1}^{N}\sum_{k,p,r}\partial_{s_{i}^{(k)}}F(S)\partial_{X_{1}^{(r)}}S_{t_{j}}^{(k)}J_{rm}\Big(\frac{X_{1}^{(p)}}{\Delta t_{1}}J_{pl}(s,X_{1})-\partial_{x^{(p)}}J_{pl}(s,X_{1})\Big)\Big]\\
 & \quad+\mathbb{E}\Big[F(S)\sum_{p}\Big(\frac{X_{1}^{(p)}}{\Delta t_{1}}\partial_{s^{(m)}}J_{pl}(s,X_{1})-\partial_{s^{(m)}x^{(p)}}^{2}J_{pl}(s,X_{1})\Big)\Big]\\
 & =\mathbb{E}\Big[\sum_{p,r}\partial_{X_{1}^{(r)}}G(X)J_{rm}(s,X_{1})\Big(\frac{X_{1}^{(p)}}{\Delta t_{1}}J_{pl}(s,X_{1})-\partial_{x^{(p)}}J_{pl}(s,X_{1})\Big)\Big]\\
 & \quad+\mathbb{E}\Big[G(X)\sum_{p}\Big(\frac{X_{1}^{(p)}}{\Delta t_{1}}\partial_{s^{(m)}}J_{pl}(s,X_{1})-\partial_{s^{(m)}x^{(p)}}^{2}J_{pl}(s,X_{1})\Big)\Big].
\end{aligned}
\]
By integration by parts,
\[
\begin{aligned} & \mathbb{E}\Big[\partial_{X_{1}^{(r)}}G(X)J_{rm}\Big(\frac{X_{1}^{(p)}}{\Delta t_{1}}J_{pl}-\partial_{x^{(p)}}J_{pl}\Big)\Big]\\
 & =\mathbb{E}\Big[G(X)\frac{X_{1}^{(r)}}{\Delta t_{1}}J_{rm}\Big(\frac{X_{1}^{(p)}}{\Delta t_{1}}J_{pl}-\partial_{x^{(p)}}J_{pl}\Big)\Big]\\
 & \quad-\mathbb{E}\Big[G(X)\partial_{x^{(r)}}\Big(J_{rm}\Big(\frac{X_{1}^{(p)}}{\Delta t_{1}}J_{pl}-\partial_{x^{(p)}}J_{pl}\Big)\Big)\Big]\\
 & =\mathbb{E}\Big[G(X)\Big(\frac{X_{1}^{(r)}}{\Delta t_{1}}J_{rm}\frac{X_{1}^{(p)}}{\Delta t_{1}}J_{pl}-\frac{X_{1}^{(r)}}{\Delta t_{1}}J_{rm}\partial_{x^{(p)}}J_{pl}\Big)\Big]\\
 & \quad-\mathbb{E}\Big[G(X)\Big(\frac{X_{1}^{(p)}}{\Delta t_{1}}\partial_{x^{(r)}}\big(J_{rm}J_{pl}\big)+\frac{\delta_{rp}}{\Delta t_{1}}J_{rm}J_{pl}-\partial_{x^{(r)}}\big(J_{rm}\partial_{x^{(p)}}J_{pl}\big)\Big)\Big].
\end{aligned}
\]
Therefore,
\[
\begin{aligned}\partial_{s^{(l)}s^{(m)}}^{2}f(s) & =\sum_{p,r}\mathbb{E}\Big[G(X)\Big(\frac{X_{1}^{(r)}}{\Delta t_{1}}J_{rm}\frac{X_{1}^{(p)}}{\Delta t_{1}}J_{pl}-\frac{X_{1}^{(r)}}{\Delta t_{1}}J_{rm}\partial_{x^{(p)}}J_{pl}\Big)\Big]\\
 & \quad-\sum_{p,r}\mathbb{E}\Big[G(X)\Big(\frac{X_{1}^{(p)}}{\Delta t_{1}}\partial_{x^{(r)}}\big(J_{rm}J_{pl}\big)+\frac{\delta_{rp}}{\Delta t_{1}}J_{rm}J_{pl}-\partial_{x^{(r)}}\big(J_{rm}\partial_{x^{(p)}}J_{pl}\big)\Big)\Big]\\
 & \quad+\mathbb{E}\Big[G(X)\sum_{p}\Big(\frac{X_{1}^{(p)}}{\Delta t_{1}}\partial_{s^{(m)}}J_{pl}(s,X_{1})-\partial_{s^{(m)}x^{(p)}}^{2}J_{pl}(s,X_{1})\Big)\Big]\\
 & =\sum_{p,r}\mathbb{E}\Big[G(X)\big(\partial_{x^{(r)}}J_{rm}\cdot\partial_{x^{(p)}}J_{pl}\big)(s,X_{1})\Big]\\
 & \quad+\sum_{p,r}\mathbb{E}\Big[G(X)\big(J_{rm}\partial_{x^{(r)}x^{(p)}}^{2}J_{pl}\big)(s,X_{1})\Big]-\sum_{p}\mathbb{E}\Big[G(X)\partial_{s^{(m)}x^{(p)}}^{2}J_{pl}(s,X_{1})\Big]\\
 & \quad+\frac{1}{\sqrt{\Delta t_{1}}}\sum_{p}\mathbb{E}\Big[G(X)Z_{1}^{(p)}\partial_{s^{(m)}}J_{pl}(s,X_{1})\Big]\\
 & \quad-\frac{1}{\sqrt{\Delta t_{1}}}\sum_{p,r}\mathbb{E}\Big[G(X)Z_{1}^{(p)}\big(\partial_{x^{(r)}}J_{rm}\cdot J_{pl}\big)(s,X_{1})\Big]\\
 & \quad-\frac{1}{\sqrt{\Delta t_{1}}}\sum_{p,r}\mathbb{E}\Big[G(X)Z_{1}^{(p)}\big(J_{rm}\partial_{x^{(r)}}J_{pl}\big)(s,X_{1})\Big]\\
 & \quad-\frac{1}{\sqrt{\Delta t_{1}}}\sum_{p,r}\mathbb{E}\Big[G(X)Z_{1}^{(r)}\big(J_{rm}\partial_{x^{(p)}}J_{pl}\big)(s,X_{1})\Big]\\
 & \quad+\frac{1}{\Delta t_{1}}\sum_{p,r}\mathbb{E}\Big[G(X)\Big(Z_{1}^{(r)}Z_{1}^{(p)}J_{rm}J_{pl}-\delta_{rp}J_{rm}J_{pl}\Big)(s,X_{1})\Big],
\end{aligned}
\]
which completes the proof of (\ref{eq:-11}).
\end{proof}

\subsection{\label{subsec:-3}Proof of Lemma \ref{lem:-2}}

We provide the proof of Lemma \ref{lem:-2} below.
\begin{proof}[Proof of Lemma \ref{lem:-2}]
(i) Since $\partial_{x}H_{0}\cdot J=\partial_{s}H_{0}$, for any
$1\le i\le n$,
\[
\sum_{p}\partial_{x^{(p)}}H_{0}^{(i)}\cdot J_{pl}=\partial_{s^{(l)}}H_{0}^{(i)}.
\]
Differentiating the above with respect to $s^{(m)}$ gives
\begin{equation}
\sum_{p}\partial_{s^{(m)}x^{(p)}}^{2}H_{0}^{(i)}\cdot J_{pl}+\sum_{p}\partial_{x^{(p)}}H_{0}^{(i)}\cdot\partial_{s^{(m)}}J_{pl}=\partial_{s^{(l)}s^{(m)}}^{2}H_{0}^{(i)}.\label{eq:-7}
\end{equation}
By $\partial_{x}H_{0}\cdot J=\partial_{s}H_{0}$ again,
\[
\sum_{r}\partial_{x^{(r)}}H_{0}^{(i)}\cdot J_{rm}=\partial_{s^{(m)}}H_{0}^{(i)},
\]
and therefore,
\begin{equation}
\sum_{r}\partial_{x^{(r)}x^{(p)}}H_{0}^{(i)}\cdot J_{rm}+\sum_{r}\partial_{x^{(r)}}H_{0}^{(i)}\cdot\partial_{x^{(p)}}J_{rm}=\partial_{s^{(m)}x^{(p)}}H_{0}^{(i)}.\label{eq:-8}
\end{equation}
Substituting (\ref{eq:-8}) to (\ref{eq:-7}) gives
\[
\partial_{s^{(l)}s^{(m)}}^{2}H_{0}^{(i)}=\sum_{p,r}\partial_{x^{(r)}x^{(p)}}H_{0}^{(i)}\cdot J_{rm}\cdot J_{pl}+\sum_{p,r}\partial_{x^{(r)}}H_{0}^{(i)}\cdot\partial_{x^{(p)}}J_{rm}\cdot J_{pl}+\sum_{p}\partial_{x^{(p)}}H_{0}^{(i)}\cdot\partial_{s^{(m)}}J_{pl}.
\]
Denote
\[
A_{lm}^{i}=\partial_{s^{(l)}s^{(m)}}^{2}H_{0}^{(i)}-\sum_{p,r}\partial_{x^{(r)}x^{(p)}}H_{0}^{(i)}\cdot J_{rm}\cdot J_{pl}.
\]
Then
\begin{equation}
A_{lm}^{i}=\sum_{p,r}\partial_{x^{(r)}}H_{0}^{(i)}\cdot\partial_{x^{(p)}}J_{rm}\cdot J_{pl}+\sum_{p}\partial_{x^{(p)}}H_{0}^{(i)}\cdot\partial_{s^{(m)}}J_{pl}.\label{eq:-10}
\end{equation}
By exchanging $p$ and $r$ in the first term on the right hand side,
\begin{equation}
A_{lm}^{i}=\sum_{p}\partial_{x^{(p)}}H_{0}^{(i)}\cdot\Big(\sum_{r}\partial_{x^{(r)}}J_{pm}\cdot J_{rl}+\partial_{s^{(m)}}J_{pl}\Big).\label{eq:-9}
\end{equation}
Let $\xi^{l,m}=\{\xi_{p}^{l,m}\}_{p}$ be the vector given by
\[
\xi_{p}^{l,m}=\sum_{r}\partial_{x^{(r)}}J_{pm}\cdot J_{rl}+\partial_{s^{(m)}}J_{pl}.
\]
Clearly, $A_{lm}^{i}=A_{ml}^{i}$. By (\ref{eq:-9}), we obtain that
\[
\sum_{p}\partial_{x^{(p)}}H_{0}^{(i)}\cdot\xi_{p}^{l,m}=\sum_{p}\partial_{x^{(p)}}H_{0}^{(i)}\cdot\xi_{p}^{m,l},\;\text{for all}\;1\le i\le n,
\]
or equivalently
\[
\partial_{x}H_{0}\cdot(\xi^{l,m}-\xi^{m,l})=0.
\]
Since $\partial_{x}H_{0}$ is invertible, we have completed the proof
of (i).

(ii) Denote
\[
\eta_{lm}=\sum_{p,r}\partial_{x^{(r)}}\big(\partial_{x^{(p)}}J_{rm}\cdot J_{pl}\big)+\sum_{p}\partial_{s^{(m)}x^{(p)}}^{2}J_{pl}.
\]
By (\ref{eq:-10}),
\[
\begin{aligned}A_{lm}^{i} & =\sum_{p,r}\partial_{x^{(r)}}H_{0}^{(i)}\cdot\partial_{x^{(p)}}J_{rm}\cdot J_{pl}+\sum_{p}\partial_{x^{(p)}}H_{0}^{(i)}\cdot\partial_{s^{(m)}}J_{pl}\\
 & =\sum_{p,r}\partial_{x^{(r)}}\big(H_{0}^{(i)}\cdot\partial_{x^{(p)}}J_{rm}\cdot J_{pl}\big)-\sum_{p,r}H_{0}^{(i)}\cdot\partial_{x^{(r)}}\big(\partial_{x^{(p)}}J_{rm}\cdot J_{pl}\big)\\
 & \quad+\sum_{p}\partial_{x^{(p)}}\big(H_{0}^{(i)}\cdot\partial_{s^{(m)}}J_{pl}\big)-\sum_{p}H_{0}^{(i)}\cdot\partial_{s^{(m)}x^{(p)}}^{2}J_{pl}\\
 & =\sum_{p,r}\partial_{x^{(r)}}\big(H_{0}^{(i)}\cdot\partial_{x^{(p)}}J_{rm}\cdot J_{pl}\big)+\sum_{p}\partial_{x^{(p)}}\big(H_{0}^{(i)}\cdot\partial_{s^{(m)}}J_{pl}\big)-H_{0}^{(i)}\eta_{lm}.
\end{aligned}
\]
Exchanging $p$ and $r$ in the first term of the last equality gives
\[
\begin{aligned}A_{lm}^{i} & =\sum_{p}\partial_{x^{(p)}}\Big(\sum_{r}H_{0}^{(i)}\cdot\partial_{x^{(r)}}J_{pm}\cdot J_{rl}\Big)+\sum_{p}\partial_{x^{(p)}}\big(H_{0}^{(i)}\cdot\partial_{s^{(m)}}J_{pl}\big)-H_{0}^{(i)}\eta_{lm}\\
 & =\sum_{p}\partial_{x^{(p)}}\big(H_{0}^{(i)}\xi_{p}^{l,m}\big)-H_{0}^{(i)}\eta_{lm}.
\end{aligned}
\]
By (i), $\xi_{p}^{l,m}=\xi_{p}^{m,l}$, which together with $A_{lm}^{i}=A_{ml}^{i}$
implies that
\[
H_{0}^{(i)}(\eta_{lm}-\eta_{ml})=0.
\]
Note that $J$ is invariant under replacement of $H_{0}^{(i)}$ by
$H_{0}^{(i)}+C$. Therefore, we may assume $H_{0}^{(i)}\not=0$ and
deduce that $\eta_{lm}=\eta_{ml}$. This, together with the $(m,l)$
symmetricity of 
\[
\sum_{p,r}\partial_{x^{(r)}}J_{rm}\partial_{x^{(p)}}J_{pl}
\]
completes the proof of (ii).

(iii) By (ii) and the $(m,l)$ symmetricity of 
\[
\sum_{p,r}\partial_{x^{(r)}}J_{rm}\partial_{x^{(p)}}J_{pl},
\]
$\Lambda^{(0)}$ is symmetric. 

To see that $\Lambda^{(1)}$ is symmetric, let
\[
\kappa_{lm}=\sum_{p,r}z^{(p)}\partial_{x^{(r)}}J_{rm}\cdot J_{pl}+\sum_{p,r}z^{(r)}J_{rm}\cdot\partial_{x^{(p)}}J_{pl}.
\]
Then
\[
\Lambda_{lm}^{(1)}+\kappa_{lm}=\sum_{p}z^{(p)}\partial_{s^{(m)}}J_{pl}-\sum_{p,r}z^{(p)}J_{rm}\partial_{x^{(r)}}J_{pl}.
\]
It is easily seen that $\kappa_{lm}=\kappa_{ml}$. Therefore
\[
\begin{aligned}\Lambda_{lm}^{(1)}-\Lambda_{ml}^{(1)} & =\sum_{p}z^{(p)}\partial_{s^{(m)}}J_{pl}-\sum_{p,r}z^{(p)}J_{rm}\partial_{x^{(r)}}J_{pl}-\sum_{p}z^{(p)}\partial_{s^{(l)}}J_{pm}+\sum_{p,r}z^{(p)}J_{rl}\partial_{x^{(r)}}J_{pm}\\
 & =\sum_{p}z^{(p)}\Big[\Big(\sum_{r}\partial_{x^{(r)}}J_{pm}\cdot J_{rl}+\partial_{s^{(m)}}J_{pl}\Big)-\Big(\sum_{r}\partial_{x^{(r)}}J_{pl}\cdot J_{rm}+\partial_{s^{(l)}}J_{pm}\Big)\Big].
\end{aligned}
\]
It follows from (i) that $\Lambda_{lm}^{(1)}=\Lambda_{ml}^{(1)}$.
This completes the proof of (iii).
\end{proof}

\subsection{\label{subsec:-4}Proof of Lemma \ref{lem:}}

We provide the proof of Lemma \ref{lem:} below.
\begin{proof}[Proof of Lemma \ref{lem:}]
(i) By Theorem \ref{thm:-1}, we may assume $F(S)=\chi_{\{\phi(S_{t_{2}},\dots,S_{t_{N}})>0\}}$.
We shall present the proof for the case where $F(S)=\chi_{\{S_{t_{2}}\in D\}}$
for some open subset $D\subset\mathbb{R}^{n}$. The proof for general
case is similar but requires more tedious notation.

Since $\Theta^{(1)}$ is a smooth function, using similar arguement
as in Theorem \ref{thm:-1}, it can be easily shown that the random
variables
\[
F(S)\Theta^{(1)}(s,Z_{1},\sqrt{\Delta t_{1}}Z_{1})-F(\hat{S})\Theta^{(1)}(s,Z_{1},0),
\]
\[
F(\hat{S})\Theta^{(1)}(s,Z_{1},\sqrt{\Delta t_{1}}\hat{Z}_{1})-F(\hat{S})\Theta^{(1)}(s,Z_{1},0),
\]
have variance of order $\mathrm{O}(\Delta t_{1})$ as $\Delta t_{1}\to0$.
Hence, we only need to prove that the random variable
\[
\frac{1}{\sqrt{\Delta t_{1}}}\big(F(S)-F(\hat{S})\big)\Theta^{(1)}(s,Z_{1},0)
\]
 has variance of order $\mathrm{O}((\Delta t_{1})^{-1/2})$. Equivalently,
we only need to prove the same for 
\[
\tilde{\xi}=\frac{1}{\sqrt{\Delta t_{1}}}\big(F(S)-F(\hat{S})\big)Z_{1}.
\]
Clearly,
\[
|\tilde{\xi}|^{2}=\frac{|Z_{1}|^{2}}{\Delta t_{1}}\chi_{\{S_{t_{2}}\in D,\hat{S}_{t_{2}}\not\in D\}\cup\{S_{t_{2}}\not\in D,\hat{S}_{t_{2}}\in D\}}\le\frac{|Z_{1}|^{2}}{\Delta t_{1}}\big(\chi_{\{S_{t_{2}}\in D,\hat{S}_{t_{2}}\not\in D\}}+\chi_{\{S_{t_{2}}\not\in D,\hat{S}_{t_{2}}\in D\}}\big).
\]
By symmetricity of $Z_{1}$,
\[
\mathbb{E}(|\tilde{\xi}|^{2})\le\frac{2}{\Delta t_{1}}\mathbb{E}\big(|Z_{1}|^{2}\chi_{\{S_{t_{2}}\in D,\hat{S}_{t_{2}}\not\in D\}}\big).
\]
The sets $\{S_{t_{2}}\in D,\hat{S}_{t_{2}}\not\in D\}$ is given by
\begin{equation}
\{H_{1}(H_{0}(s,\sqrt{\Delta t_{1}}Z_{1}),\sqrt{t_{2}-\Delta t_{1}}Z_{2})\in D,H_{1}(H_{0}(s,-\sqrt{\Delta t_{1}}Z_{1}),\sqrt{t_{2}-\Delta t_{1}}Z_{2})\not\in D\}.\label{eq:-34}
\end{equation}
For each $s$, let $I(s,\cdot)$ be the inverse of the map $x\mapsto H_{1}(s,x)$.
That is, $I(s,H_{1}(s,x))=x$. The set (\ref{eq:-34}) can be written
as 
\[
\{(Z_{1},Z_{2}):\sqrt{t_{2}-\Delta t_{1}}Z_{2}\in I(H_{0}(s,\sqrt{\Delta t_{1}}Z_{1}),D)\cap I(H_{0}(s,-\sqrt{\Delta t_{1}}Z_{1}),\mathbb{R}^{n}\backslash D).
\]
For any $x\in\mathbb{R}^{n}$ with $|x|\le1$, since $I$ and $H_{0}$
are differentiable, the probability
\[
p(x)=\mathbb{P}(\sqrt{t_{2}-\Delta t_{1}}Z_{2}\in I(H_{0}(s,x),D)\cap I(H_{0}(s,-x),\mathbb{R}^{n}\backslash D))
\]
is differentiable in $x$, and is zero at $x=0$. Therefore, 
\[
p(x)\le C|x|,\;\text{for all}\;|x|\le1,
\]
where $C>0$ is a constant which can be chosen to be independent of
$\Delta t_{1}$. Hence,
\[
\begin{aligned}\mathbb{E}(|\tilde{\xi}|^{2}) & \le\frac{2}{\Delta t_{1}}\mathbb{E}\big(|Z_{1}|^{2}\chi_{\{S_{t_{2}}\in D,\hat{S}_{t_{2}}\not\in D\}}\big)\\
 & =\frac{2}{\Delta t_{1}}\mathbb{E}(|Z_{1}|^{2}p(\sqrt{\Delta t_{1}}Z_{1}))\\
 & \le\frac{2}{\Delta t_{1}}\Big(C\sqrt{\Delta t_{1}}\mathbb{E}(|Z_{1}|^{3}\chi_{\{\sqrt{\Delta t_{1}}|Z_{1}|\le1\}})+\mathbb{E}(|Z_{1}|^{2}\chi_{\{\sqrt{\Delta t_{1}}|Z_{1}|>1\}})\Big)\\
 & \le\frac{2}{\Delta t_{1}}\Big(C\sqrt{\Delta t_{1}}\mathbb{E}(|Z_{1}|^{3})+\mathbb{E}(|Z_{1}|^{2}\chi_{\{\sqrt{\Delta t_{1}}|Z_{1}|>1\}})\Big)\\
 & \le\frac{C}{\sqrt{\Delta t_{1}}}+\frac{C}{\Delta t_{1}}\int_{1/\sqrt{\Delta t_{1}}}^{\infty}r^{2}e^{-\frac{r^{2}}{2}}\cdot r^{n-1}dr\\
 & \le\frac{C}{\sqrt{\Delta t_{1}}}+\frac{C}{\Delta t_{1}}e^{-\frac{1}{4\Delta t_{1}}}\\
 & \le\frac{C}{\sqrt{\Delta t_{1}}},
\end{aligned}
\]
as $\Delta t_{1}\to0$, where, by abuse of notation, the constant
$C>0$ may vary from line to line. This completes the proof of (\ref{eq:-40}). 

The proof of (\ref{eq:-41}) is similar to that of (\ref{eq:-40}).

For the proof of (\ref{eq:-42}), similar to the above, we only need
to prove the $\text{O}((\Delta t_{1})^{-3/2})$ variance for the random
variable
\[
\frac{1}{\Delta t_{1}}\big(F(S)-2F(\bar{S})+F(\hat{S})\big)\Lambda^{(2)}(s,Z_{1},0).
\]
Equivalently, we only show that the variance of 
\[
\tilde{\eta_{2}}=\frac{1}{\Delta t_{1}}\big(F(S)-2F(\bar{S})+F(\hat{S})\big)(Z_{1}^{\text{T}}Z_{1}-I)
\]
is of order $\text{O}((\Delta t_{1})^{-3/2})$ as $\Delta t_{1}\to0$.
Using similar argument as before, it can be shown that
\[
\mathbb{E}\big(|F(S)-2F(\bar{S})+F(\hat{S})|^{2}\big|Z_{1}\big)=\mathbb{E}\big(|\chi_{\{S_{t_{2}}\in D\}}-2\chi_{\{\bar{S}_{t_{2}}\in D\}}+\chi_{\{\hat{S}_{t_{2}}\in D\}}|^{2}\big|Z_{1}\big)\le C\sqrt{\Delta t_{1}}|Z_{1}|
\]
for all $Z_{1}$ with $|Z_{1}|\le1/\sqrt{\Delta t_{1}}$. Therefore,
\[
\begin{aligned}\mathbb{E}(|\tilde{\eta_{2}}|^{2}) & \le\frac{1}{(\Delta t_{1})^{2}}\Big(C\sqrt{\Delta t_{1}}\mathbb{E}\big((1+|Z_{1}|)^{5}\chi_{\{\sqrt{\Delta t_{1}}|Z_{1}|\le1\}}\big)+\mathbb{E}\big((1+|Z_{1}|)^{4}\chi_{\{\sqrt{\Delta t_{1}}|Z_{1}|>1\}}\big)\Big)\\
 & \le\frac{C}{(\Delta t_{1})^{2}}\Big(\sqrt{\Delta t_{1}}\mathbb{E}\big((1+|Z_{1}|)^{5}\big)+\mathbb{E}\big(|Z_{1}|^{4}\chi_{\{\sqrt{\Delta t_{1}}|Z_{1}|>1\}}\big)\Big)\\
 & \le\frac{C}{(\Delta t_{1})^{3/2}}+\frac{C}{(\Delta t_{1})^{2}}\int_{1/\sqrt{\Delta t_{1}}}^{\infty}r^{4}e^{-\frac{r^{2}}{2}}\cdot r^{n-1}dr\\
 & \le\frac{C}{(\Delta t_{1})^{3/2}}+\frac{C}{(\Delta t_{1})^{2}}e^{-\frac{1}{4\Delta t_{1}}}\\
 & \le\frac{C}{(\Delta t_{1})^{3/2}}.
\end{aligned}
\]
This completes the proof of (\ref{eq:-42}).

(ii) Equations (\ref{eq:-60}) and (\ref{eq:-61}) are the result
of Theorem \ref{thm:-1}. For equation (\ref{eq:-46}), without loss
of generality, we may asume $F=F^{c}$ and $\partial F^{c}=\chi_{\{\phi(S_{t_{2}},\dots,S_{t_{N}})>0\}}$.
Using the same notation as in (i), by H\textcyr{\"\cyro}lder's inequality,
\begin{align*}
|\tilde{\eta_{2}}|^{2} & \le\frac{C(1+|Z_{1}|)^{5}}{(\Delta t_{1})^{2}}\Big|\int_{0}^{\sqrt{\Delta t_{1}}}\partial_{x_{1}}K(\lambda Z_{1},\sqrt{\Delta t_{2}}Z_{2},\dots,\sqrt{\Delta t_{N}}Z_{N})-\partial_{x_{1}}K(-\lambda Z_{1},\sqrt{\Delta t_{2}}Z_{2},\dots,\sqrt{\Delta t_{N}}Z_{N})d\lambda\Big|^{2}\\
 & \le\frac{C(1+|Z_{1}|)^{5}}{(\Delta t_{1})^{3/2}}\int_{0}^{\sqrt{\Delta t_{1}}}\big|\partial_{x_{1}}K(\lambda Z_{1},\sqrt{\Delta t_{2}}Z_{2},\dots,\sqrt{\Delta t_{N}}Z_{N})-\partial_{x_{1}}K(-\lambda Z_{1},\sqrt{\Delta t_{2}}Z_{2},\dots,\sqrt{\Delta t_{N}}Z_{N})\big|^{2}d\lambda.
\end{align*}
Therefore,
\[
\mathbb{E}(|\tilde{\eta_{2}}|^{2})\le\frac{C}{(\Delta t_{1})^{3/2}}\int_{0}^{\sqrt{\Delta t_{1}}}\mathbb{E}\big[(1+|Z_{1}|)^{5}\big|\partial_{x_{1}}K(\lambda Z_{1},\sqrt{\Delta t_{2}}Z_{2},\dots,\sqrt{\Delta t_{N}}Z_{N})-\partial_{x_{1}}K(-\lambda Z_{1},\sqrt{\Delta t_{2}}Z_{2},\dots,\sqrt{\Delta t_{N}}Z_{N})\big|^{2}\big]d\lambda.
\]
By the same argument as in (i), it can be shown that
\[
\mathbb{E}\big[(1+|Z_{1}|)^{5}\big|\partial_{x_{1}}K(\lambda Z_{1},\sqrt{\Delta t_{2}}Z_{2},\dots,\sqrt{\Delta t_{N}}Z_{N})-\partial_{x_{1}}K(-\lambda Z_{1},\sqrt{\Delta t_{2}}Z_{2},\dots,\sqrt{\Delta t_{N}}Z_{N})\big|^{2}\big]\le C\lambda.
\]
Therefore,
\[
\mathbb{E}(|\tilde{\eta_{2}}|^{2})\le\frac{C}{(\Delta t_{1})^{3/2}}\int_{0}^{\sqrt{\Delta t_{1}}}\lambda d\lambda=\frac{C}{\sqrt{\Delta t_{1}}}.
\]
This proves (\ref{eq:-46}).
\end{proof}

\subsection{\label{subsec:-5}Proof of Theorem \ref{thm:-2}}

We provide the proof of Theorem \ref{thm:-2} below. We start from
the following bias estimate.
\begin{lem}
\label{lem:-1}Let $X$, $\tilde{X}$ be two $n$-dimensional random
variables such that $\mathbb{E}(X)=\mathbb{E}(\tilde{X})=x_{0}$,
and $g(x)\in C^{2}(\mathbb{R}^{n})$. Let
\begin{equation}
\kappa(r)=\max_{|x|\le r}|\partial_{x}^{2}g(x)|,\;r>0.\label{eq:-3}
\end{equation}
Then
\[
\big|\mathbb{E}[g(\tilde{X})]-\mathbb{E}[g(X)]\big|\le C\mathbb{E}\big[\big(\kappa(|x_{0}|)+\kappa(|X|)+\kappa(|\tilde{X}|)\big)(|\tilde{X}-X|^{2}+|X-x_{0}||\tilde{X}-X|)\big].
\]
\end{lem}

\begin{proof}
By Talor expansion,
\[
g(\tilde{X})-g(X)=\partial_{x}g(X)(\tilde{X}-X)+\frac{1}{2}(\tilde{X}-X)^{\text{T}}\partial_{x}^{2}g(\theta X+(1-\theta)\tilde{X})(\tilde{X}-X),
\]
for some $\theta\in(0,1)$. By (\ref{eq:-3}), we have
\begin{equation}
\big|(\tilde{X}-X)^{\text{T}}\partial_{x}^{2}g(\theta X+(1-\theta)\tilde{X})(\tilde{X}-X)\big|\le\big(\kappa(|X|)+\kappa(|\tilde{X}|)\big)|\tilde{X}-X|^{2}.\label{eq:-4}
\end{equation}
By $\mathbb{E}(X)=\mathbb{E}(\tilde{X})$,
\begin{equation}
\begin{aligned}\mathbb{E}\big[\partial_{x}g(X)(\tilde{X}-X)\big] & =\mathbb{E}[\partial_{x}g(x_{0})(\tilde{X}-X)\big]+\mathbb{E}\big[(\partial_{x}g(X)-\partial_{x}g(x_{0}))(\tilde{X}-X)\big]\\
 & =\mathbb{E}\big[(\partial_{x}g(X)-\partial_{x}g(x_{0}))(\tilde{X}-X)\big].
\end{aligned}
\label{eq:-6}
\end{equation}
Note that, by Taylor expansion again,
\[
|\partial_{x}g(X)-\partial_{x}g(x_{0})|\le\big(\kappa(|x_{0}|)+\kappa(|X|)\big)|X-x_{0}|.
\]
Therefore,
\begin{equation}
\begin{aligned}\big|\mathbb{E}\big[(\partial_{x}g(X)-\partial_{x}g(x_{0}))(\tilde{X}-X)\big]\big| & \le\mathbb{E}\big[\big(\kappa(|x_{0}|)+\kappa(|X|)\big)|X-x_{0}||\tilde{X}-X|\big]\end{aligned}
\label{eq:-47}
\end{equation}
Combining (\ref{eq:-4}), (\ref{eq:-6}), and (\ref{eq:-47}) completes
the proof.
\end{proof}
We can now prove Theorem \ref{thm:-2}.
\begin{proof}[Proof of Theorem \ref{thm:-2}]
Clearly,
\[
\mathbb{E}[F(S_{t_{2}},\dots,S_{t_{N}})|S_{t_{1}}=\theta]=\mathbb{E}[F(\tilde{S}_{t_{2}},\dots,\tilde{S}_{t_{N}})|\tilde{S}_{t_{1}}=\theta].
\]
Let
\[
g(x)=\mathbb{E}[F(S_{t_{2}},\dots,S_{t_{N}})|S_{t_{1}}=H(t_{1},\sqrt{\Delta t_{1}},s,x)].
\]
Then
\begin{equation}
\mathbb{E}(F(S_{t_{2}},\dots,S_{t_{N}}))-\mathbb{E}(F(\tilde{S}_{t_{2}},\dots,\tilde{S}_{t_{N}}))=\mathbb{E}(g(\sqrt{\Delta t_{1}}Z_{1}))-\mathbb{E}(g(\sqrt{\Delta t_{1}}Z_{1}+\epsilon\sqrt{\Delta t_{1}}Y)).\label{eq:-48}
\end{equation}
By applying (\ref{eq:-11}) to $S_{t}$ conditional on $S_{t_{1}}$,
we obtain that $g(x)\in C^{2}(\mathbb{R}^{n})$ and
\[
|\partial_{x}^{2}g(x)|\le C_{1}e^{C_{2}|x|},
\]
where the constants $C_{1}>0$ and $C_{2}>0$ can be chosen as independent
of $\Delta t_{1}$\footnote{Note that, for the conditional process, the $\Delta t_{1}$ in (\ref{eq:-11})
becomes $\Delta t_{2}$ here.}. By Lemma \ref{lem:-1}, 
\begin{align*}
 & \big|\mathbb{E}(g(\sqrt{\Delta t_{1}}Z_{1}))-\mathbb{E}(g(\sqrt{\Delta t_{1}}Z_{1}+\epsilon\sqrt{\Delta t_{1}}Y))\big|\\
 & \le C\mathbb{E}\big[(e^{C_{2}\sqrt{\Delta t_{1}}|Z_{1}|}+e^{C_{2}\epsilon\sqrt{\Delta t_{1}}|Y|})(\epsilon\Delta t_{1}|Y||Z_{1}|+\epsilon^{2}\Delta t_{1}|Y|^{2})\big]\\
 & \le C\epsilon\Delta t_{1}\mathbb{E}\big[(e^{C_{2}\sqrt{\Delta t_{1}}|Z_{1}|}+e^{C_{2}\epsilon\sqrt{\Delta t_{1}}|Y|})(|Z_{1}|^{2}+|Y|^{2})\big].
\end{align*}
Since 
\[
\mathbb{E}(e^{\lambda|Y|})<\infty,\mathbb{E}(e^{\lambda|Z_{1}|})<\infty,
\]
we obtain that
\[
\mathbb{E}\big[(e^{C_{2}\sqrt{\Delta t_{1}}|Z_{1}|}+e^{C_{2}\epsilon\sqrt{\Delta t_{1}}|Y|})(|Z_{1}|^{2}+|Y|^{2})\big]<C<\infty,
\]
for some constant $C>0$ independent of $\epsilon$. Therefore,
\begin{equation}
\big|\mathbb{E}(g(\sqrt{\Delta t_{1}}Z_{1}))-\mathbb{E}(g(\sqrt{\Delta t_{1}}Z_{1}+\epsilon\sqrt{\Delta t_{1}}Y))\big|\le C\epsilon\Delta t_{1}.\label{eq:-49}
\end{equation}
Combining (\ref{eq:-48}) and (\ref{eq:-49}) completes the proof.
\end{proof}


\begin{thebibliography}{1}
\bibitem{Gla04}Glasserman, P. (2004). Monte Carlo Methods in Financial
Engineering.

\end{thebibliography}
\end{document}